\RequirePackage{algorithmicx}

\documentclass{siamart250211}

\usepackage{booktabs}
\usepackage[sort,noadjust]{cite}
\usepackage[siamart]{lindquist}

\usepackage{acro}
\acsetup{barriers/use=true,barriers/reset=true}

\DeclareAcronym{genp}{short=GENP,long=Gaussian elimination with no pivoting}
\DeclareAcronym{gepp}{short=GEPP,long=Gaussian elimination with partial pivoting}
\DeclareAcronym{getp}{short=GETP,long=Gaussian elimination with tournament pivoting}
\DeclareAcronym{gethp}{short=GEThP,long=Gaussian elimination with threshold pivoting}
\DeclareAcronym{beam}{short=BEAM,long=block elimination with additive modifications}

\DeclareAcronym{gmres}{short=GMRES,long=generalized minimal residual method}
\DeclareAcronym{fgmres}{short=FGMRES,long=flexible \acs{gmres}}
\DeclareAcronym{cbgmres}{short=CB-GMRES,long=compressed basis \acs{gmres}}
\DeclareAcronym{cg}{short=CG,long=conjugate gradients method}

\DeclareAcronym{cgs}{short = CGS, long = classical Gram-Schmidt}
\DeclareAcronym{mgs}{short = MGS, long = modified Gram-Schmidt}
\DeclareAcronym{cgsr}{short = CGS2, long = classical Gram-Schmidt with reorthogonalization}

\DeclareAcronym{rbt}{short=RBT, long=random butterfly transform}
\DeclareAcronym{srft}{short=SRFT, long=subsampled random Fourier transform}

\DeclareAcronym{svd}{short=SVD, long=singular value decomposition}
\DeclareAcronym{dft}{short=DFT, long=discrete Fourier transform}
\DeclareAcronym{fft}{short=FFT, long=fast Fourier transform}

\DeclareAcronym{blr}{short=BLR, long=block low-rank}
\DeclareAcronym{aca}{short=ACA, long=adaptive cross approximation}

\DeclareAcronym{csr}{short=CSR, long=compressed sparse row}
\DeclareAcronym{ilu}{short=ILU,long=incomplete LU}

\DeclareAcronym{magma}{short=MAGMA,long=Matrix Algebra on GPU and Multicore Architectures}
\DeclareAcronym{plasma}{short=PLASMA,long=Parallel Linear Algebra Software for Multicore Architectures}
\DeclareAcronym{slate}{short=SLATE,long=Software for Linear Algebra Targeting Exascale}
\DeclareAcronym{ecp}{short=ECP,long=Exascale Computing Project}
\DeclareAcronym{dag}{short=DAG, long=direct acyclic graph}
\DeclareAcronym{hbm}{short=HBM, long=high bandwidth memory}
\DeclareAcronym{hbm2}{short=HBM2, long=high bandwidth memory}

\usepackage{pgfplots}
\usepgfplotslibrary{statistics}

\usepackage{algpseudocode}

\newcommand{\Rho}{\mathrm{P}}
\newcommand{\blockgrowth}[2][]{\Rho^{#1}_{\!#2}}
\newcommand{\uniblockgrowth}[2][1]{\Rho^{\langle #1\rangle}_{#2}}
\newcommand{\maxrm}{\mathrm{max}}

\newcommand{\modmat}[2][]{M_{#2}^{#1}}

\newif\iffixlabel
\IfFileExists{/Users/luszczek/.zshrx}{\fixlabeltrue}{\fixlabelfalse}
\iffixlabel
\newcommand{\fixlabel}[1]{\label{{#1}}}
\else
\newcommand{\fixlabel}[1]{\label{#1}}
\fi

\title{The Stability\\of Block Eliminations\\and Additive Modifications%
\thanks{Most of the analysis in the start of \cref{sec:block-lu} and in \cref{sec:block-lu:growth,sec:block-lu:blockwise} previously appeared in the first author's dissertation~\cite[Sec.~3.4.3]{lindquistReducingCommunicationSolution2023}.
\funding{This work was funded in part by the National Science Foundation under grant no.~2004541 and
the Exascale Computing Project, a collaborative effort of the U.S.
Department of Energy Office of Science and National Nuclear Security
Administration. \\
Research was sponsored by the Department of the Air Force Artificial Intelligence Accelerator and was accomplished under Cooperative Agreement Number FA8750-19-2-1000. The views and conclusions contained in this document are those of the authors and should not be interpreted as representing the official policies, either expressed or implied, of the Department of the Air Force or the U.S. Government. The U.S. Government is authorized to reproduce and distribute reprints for Government purposes notwithstanding any copyright notation herein.}}}

\author{Neil Lindquist\thanks{Innovative Computing Laboratory, University of Tennessee, Knoxville, TN, USA 37996 (\email{nlindquist@acm.org}, \email{dongarra@icl.utk.edu)}}
\and Piotr Luszczek\thanks{MIT Lincoln Laboratory, Lexington, MA, USA 02421 and Innovative Computing Laboratory, University of Tennessee, Knoxville, TN, USA 37902 (\email{luszczek@icl.utk.edu})}
\and Jack Dongarra\footnotemark[2] 
}

\headers{Block Eliminations and Additive Modifications}{Neil Lindquist, Piotr Luszczek, and Jack Dongarra}

\begin{document}

\maketitle

\begin{abstract}
The \ac{beam} method was recently proposed as an alternative to LU with partial pivoting requiring less communication.
Because of the novelty of \ac{beam}, the existing theoretical analysis is lacking.
To that end, we analyze both the numerical stability of the underlying block LU factorization and the effects of additive modifications.
For the block LU factorization, we are able to improve the previous results of Demmel et al.\ from being cubic in the element growth to merely quadratic.
Furthermore, we propose an alternative measure of element growth that is better aligned with block LU; this new measure of growth allows our analysis to apply to matrices that cannot be factored with pointwise LU.
In the second part, we analyzed the modifications produced by \ac{beam} and the effect they have on the condition number and growth factor.
Finally, we show that \ac{beam} will not apply any modifications in some cases that regular block LU can safely factor.
\end{abstract}
\begin{keywords}
Block LU, numerical stability, growth factor, LU factorization, Gaussian elimination
\end{keywords}
\begin{MSCcodes}
15A23, 65F05
\end{MSCcodes} 

\section{Introduction}
Large, dense systems of linear equations are commonly solved using LU with partial pivoting; however the communication is increasingly a bottleneck on modern supercomputers.
Recently, we proposed \acf{beam} as a cheaper alternative to partial
pivoting~\cite{lindquistUsingAdditiveModifications2023}.
This method is able to achieve high performance by using the parallel dependencies of LU without pivoting,
but it achieves better numerical stability by factoring the diagonal blocks with an \acs{svd} and modifying diagonal blocks with small singular values.
Clearly, the structure of the L and U factors had to be changed to
become block triangular matrices.
However, there are several open questions about the numerical stability of the method, including the backward stability of \ac{beam}'s formulation of block LU and the effect of additive modifications.
We study these questions to provide a better understanding of the stability of \ac{beam}.

We begin by providing an outline of the \ac{beam} algorithm in \cref{sec:intro-beam}.
Then, \cref{sec:block-lu} analyzes block LU with a focus on the form that occurs in \ac{beam}.
We first propose a generalized definition of the growth factor in \cref{sec:block-lu:growth} that better supports the blocked nature of this factorization;
this growth factor is used in \cref{sec:block-lu:blockwise} to improve the backward error bound of block LU.
The analysis is continued in \cref{sec:block-lu:diag-dom} by showing that block LU is numerically stable on both diagonally dominant and block-diagonally dominant matrices.
Finally, \cref{sec:block-lu:experiments} reinforces some of the key results on block LU experimentally.

After having considered the stability of block LU, we consider the effects of \ac{beam}'s additive modifications on the condition number and growth factor in \cref{sec:svals}.
Finally, \cref{sec:mod-free-beam} explores the conditions for which \ac{beam} does not apply any
modifications and show that H-matrices and symmetric positive definite matrices
will not receive modifications for a sufficiently small tolerance \(\tau\).

%

\subsection{BEAM algorithm}
\label{sec:intro-beam}

The key idea of \ac{beam} is to apply a non-pivoted block LU, using the SVD to factor the diagonal blocks.
Then, if any of the singular values are below a preset tolerance, \(\tau\), they are modified to make the leading diagonal blocks less ill-conditioned.
Note that if the block size is sufficiently small, the increased computational cost of using the SVD for the diagonal blocks will be negligible compared to the overall cost of the solver.
Furthermore, avoiding the need to pivot significantly reduces the amount of data movement.
This allows \ac{beam} to outperform LU with partial pivoting or tournament pivoting for large dense matrices on distributed, heterogeneous computing clusters~\cite{lindquistUsingAdditiveModifications2023}.

For the sake of generality, we define the blocks based on the
monotonically increasing list of the (element-wise) row or column for which
each block starts, called \(\mathcal{I}\).
For notational convenience, \(n+1\) is added as the last element of
\(\mathcal{I}\) to act as ``past-the-end`` marker.
For example, a fixed block size, \(n_b\), has blocks described by \(\mathcal{I} = [1, n_b+1, 2n_b+1, \dots, (n_t-1)n_b+1, n+1]\) where \(n_t = \ceil{n/n_b}\) and using 1-based indexing.
For simplicity, we index blocks of a matrix with subscripts, as exemplified by
\[
	A_{i,j:k} = A[\mathcal{I}_{i}:\mathcal{I}_{i+1}\mathord{-}1,\, \mathcal{I}_{j}:\mathcal{I}_{k-1}\mathord{-}1]
\]
where \(\mathcal{I}_{i}\) is the \(i\)th element of \(\mathcal{I}\).
Note that to simplify notation, we index the blocks of the intermediate Schur-complements, \(A^{(k)}\), from \(k\) to \(n_t\) instead of \(1\) to \(n_t-k+1\).

\Cref{alg:beam} outlines the \ac{beam} algorithm.
\begin{algorithm}[tpb]
\caption{\Acf{beam}}
\fixlabel{alg:beam}
\begin{algorithmic}[1]
  \Procedure{\textsf{Factor\ac{beam}}}{$A, \tau$}
    \State \(n_t \gets \) the number of blocks in \(\mathcal{I}\)
    \State \(m \gets 0\) \Comment{Counter for the number of modifications}
    \State \(A^{(1)} \gets A\)
    \For{\(k = 1:n_t\)}
        \State \(U_{k}, \Sigma_{k}, V_{k}^T \gets \mathrm{SVD}(A_{k,k}^{(k)})\)
        \For{\(i = 1:n_b\)}
              \If{\(\Sigma_{k}[i] \leq \tau\)} \Comment{If the singular value is too small}
                \State \(m \gets m + 1\) \Comment{Record modification}
                \State \(\modmat{\Sigma}[m, m] \gets \tau - \Sigma_{k}[i]\)
                \State \(\modmat{U}[:, m] \gets [0, U_k[:, i]^T, 0]^T\)
                \State \(\modmat{V}[:, m] \gets [0, V_k[:, i]^T, 0]^T\)
                \State \(\Sigma_{k}[i] \gets \tau\) \Comment{Apply modification}
              \EndIf
        \EndFor
        \State \(\widetilde{L}_{k,k} \gets U_{k}\)
        \State \(\widetilde{R}_{k,k} \gets \Sigma_{k}V_{k}^T\)
        \State \(\widetilde{L}_{k+1:n_t, k} \gets A_{k+1:n_t, k}\,\widetilde{R}_{k,k}^{-1}\)
        \State \(\widetilde{R}_{k, k+1:n_t} \gets \widetilde{L}_{k,k}^{-1}\,A_{k, k+1:n_t}\)
        \State \(A_{k+1:n_t, k+1:n_t}^{(k+1)} \gets A_{k+1:n_t, k+1:n_t}^{(k)} -  \widetilde{L}_{k+1:n_t, k}\,\widetilde{R}_{k, k+1:n_t}\)
    \EndFor
    \If{\(m > 0\) and using Woodbury formula}
        \State \(\mathcal{C}_R \gets \modmat{\Sigma}\modmat[T]{V}\widetilde{R}^{-1}\)
        \State \(\mathcal{C}_L \gets \widetilde{L}^{-1}\modmat{U}\)
        \State \(\mathcal{C} \gets I - \mathcal{C}_R\mathcal{C}_L\)
        \State \(\mathcal{C}^{-1} \gets \textsc{Factor}(\mathcal{C})\)
            \Comment{E.g., \ac{gepp}}
    \EndIf
  \EndProcedure
  \Procedure{\textsf{Solve\ac{beam}}}{$b$}
    \State \(x \gets \widetilde{L}^{-1}b\)
    \If{\(m > 0\) and using Woodbury formula}
        \State \(x \gets (I + \mathcal{C}_L\mathcal{C}^{-1}\mathcal{C}_R) x\)
    \EndIf
    \State \(x \gets \widetilde{R}^{-1}x\)
  \EndProcedure
\end{algorithmic}
\end{algorithm}
The user-provided input \(\tau = \hat{\tau}\norm{2}{A}\) controls how
small the singular values of the diagonal blocks can be before the
modifications are applied to increase them for better
numerical stability.
This algorithm is equivalent to factoring the modified matrix \({\widetilde{A} = A + M_UM_\Sigma M_V^T}\) with block LU~\cite{lindquistUsingAdditiveModifications2023}.
The modification is represented by the product \(M_UM_\Sigma M_V^T\) where \(M_U\) and \(M_V^T\) are unitary block-diagonal and \(M_\Sigma\) is diagonal with elements being positive but less than \(\tau\).
Note that we label the factors as \(L\) and \(R\) instead of the traditional \(L\) and \(U\);
this avoids potential confusion with the matrix of left-singular vectors, which is also traditionally denoted \(U\).
Furthermore, unlike many formulations of block LU, we allow the diagonal blocks of \(L\) to be general unitary matrices, instead of only the identity matrix.

\Ac{beam} can optionally use the Woodbury formula to correct the modifications in the solve step.
When the Woodbury formula is not used, the modifications must be corrected by iterative refinement, although even with the Woodbury formula, iterative refinement is usually needed to correct for the larger growth factor incurred by \ac{beam}~\cite{lindquistUsingAdditiveModifications2023}.
The computational cost of the Woodbury formula obviously depends on the number of modifications to be corrected.
Furthermore, previous experimental evidence suggests that, for ill-conditioned matrices, using the Woodbury formula is significantly more effective than relying solely on iterative refinement~\cite{lindquistUsingAdditiveModifications2023}.

\subsection{Matrix norms and their propeties}
\label{sec:intro-norms}
Strictly speaking, a particular matrix norm is only defined for a single matrix size.
Because the analysis of \ac{beam} and block LU requires working with submatrices of various sizes, we instead consider dimension-invariant families of matrix norms.
Such families are sets of matrix norms such that for any matrix, \(A\), and any partitioning thereof, \(\mathcal{I}\times \mathcal{J}\), the norms of appropriate size satisfy
\begin{equation}
	\label{eq:norm-families}
	\max_{\substack{i\in\mathcal{I}\\j\in\mathcal{J}}} \norm{}{A_{i,j}}
	\leq \norm{}{A}
	\leq \sum_{\substack{i\in\mathcal{I}\\j\in\mathcal{J}}} \norm{}{A_{i,j}}.
\end{equation}
This is satisfied by all of the usual norms, including those induced by vector \(\ell_p\) norms, the Schatten norms, and the elementwise norms.
A key result of this definition is that padding a matrix with zero rows or columns does not affect its norm.
We use dimension-invariant families of matrix norms in all further contexts involving matrices of multiple sizes without further comment.

Additionally, we recall two key norm properties.
A matrix norm, \(\norm{\alpha}{\cdot}\), is submultiplicative if for all conformal matrices \(A\) and \(B\), we have \(\norm{\alpha}{AB} \leq \norm{\alpha}{A}\norm{\alpha}{B}\).
Note that some authors limit the term ``matrix norm'' to submultiplicative norms;
we make no such limitation, due to the use of the max-norm in previous analyses~\cite{wilkinsonRoundingErrorsAlgebraic1963,demmelStabilityBlockLU1995}.
Submultiplicative norms include all matrix norms induced by vector \(\ell_p\) norms and the Frobenius norm.
A matrix norm, \(\norm{\alpha}{\cdot}\), is absolute if for any matrix \(A\), we have \(\norm{\alpha}{|A|}=\norm{\alpha}{A}\).
Absolute norms include the max-norm, both the 1- and \(\infty\)-operator norms, and the Frobenius norm.

Due to the block structure of factorizations studied here, it is useful
to define two related matrix norms induced on the same blocking factors with \(n_t\) being the number of block rows and \(m_t\) being the number of block columns.
For an inner norm \(\norm{\alpha}{\cdot}\), we define the block-max norm and the block-sum norm as
\[
	\norm{\max\alpha}{B} = \max_{\substack{1\leq i\leq m_t\\1\leq j\leq n_t}} \norm{\alpha}{B_{ij}}
	\qquad\text{and}\qquad
	\norm{\Sigma\alpha}{B} = \sum_{i=1}^{m_t}\sum_{j=1}^{n_t} \norm{\alpha}{B_{ij}}.
\]
These norms have several useful properties:
\begin{itemize}
\item \(m_t^{-1}n_t^{-1}\norm{\Sigma\alpha}{B} \leq \norm{\max\alpha}{B}
  \leq \norm{\alpha}{B} \leq \norm{\Sigma\alpha}{B} \leq m_tn_t
    \norm{\max\alpha}{B}\).
\item The block-sum norm is submultiplicative if and only if the inner norm is. But, the block-max norm is never submultiplicative.
\item Each block norm is absolute if and only if the inner norm is absolute.
\item Each block norm is invariant to transposition (or conjugate-transposition) if and only if the inner norm is also invariant.
\item If the blocks of a matrix, \(B\), are partitioned column-wise,
\(B = [B_1, B_2]\), then \(\norm{\max\alpha}{B} \!=
    \max(\norm{\max\alpha}{B_1}, \norm{\max\alpha}{B_2})\) and \(\norm{\Sigma\alpha}{B} \!= \norm{\Sigma\alpha}{B_1} + \norm{\Sigma\alpha}{B_2}\).  Likewise for row-wise partitioning.
These norm properties of partitioning-invariance proves useful in
establishing the numerical properties of block algorithms.
\end{itemize}
Finally, the norm-equivalencies can be tightened for specific norms.
For example,
\[
	\norm{1}{B} \leq m_t\norm{\max1}{B}, \qquad \norm{\infty}{B}
        \leq n_t\norm{\max\infty}{B},  \qquad\text{and}
\]\[
	m_t^{-1/2}n_t^{-1/2}\norm{\Sigma F}{B} \leq \norm{F}{B} \leq
        m_t^{1/2}n_t^{1/2}\norm{\max F}{B}.
\]

\section{Block LU Factorization}
\label{sec:block-lu}

To our knowledge, the best perturbation analysis for block LU is that of Demmel et al.~\cite{demmelStabilityBlockLU1995}.
Unfortunately, they limit their analysis to one particular formulation, with the diagonal blocks being factored into \(I\) and \(A_{ii}\).
We improve on their analysis, both by generalizing how the diagonal
blocks are factored and by tightening the bound through the use of
alternative measures of element growth (described in
\cref{sec:block-lu:growth}). Our generalized formulation of the LU also
applies to Algs.~13.1 and 13.2 of Higham~\cite{highamAccuracyStabilityNumerical2002}.
Another interesting work in this area is by Dopico and Molera~\cite{dopicoPerturbationTheoryFactorizations2005}, in which they provide forward error bounds on \(L\) and \(U\) given a bound on the error \(A - \widehat{L}\widehat{U}\).

\begin{algorithm}[tpb]
\caption{Block LU factorization of \(A\)}
\fixlabel{alg:block-lu}
\begin{algorithmic}[1]
	\State \(n_t \gets \) the number of blocks in \(\mathcal{I}\)
	\State \(A^{(1)} \gets A\)
	\For{\(k = 1:n_t\)}
		\State \(L_{k,k}, R_{k,k} \gets \textsc{Factor}(A^{(k)}_{k,k}))\)
			\(\fixlabel{alg:block_lu:diag}\)
		\State \(L_{k+1:n_t, k} \gets A^{(k)}_{k+1:n_t, k}\,R_{k,k}^{-1}\)
		\State \(R_{k, k+1:n_t} \gets L_{k,k}^{-1}\,A^{(k)}_{k, k+1:n_t}\)
		\State \(A^{(k+1)}_{k+1:n_t, k+1:n_t} \gets A^{(k )}_{k+1:n_t, k+1:n_t} -  L_{k+1:n_t, k}\,R_{k, k+1:n_t}\)
			\(\fixlabel{alg:block_lu:schur-comp}\)
	\EndFor
\end{algorithmic}
\end{algorithm}

To analyze block LU, we focus on the general algorithm described by \cref{alg:block-lu};
this is the same basic algorithm as pointwise non-pivoted LU, except matrix operations replace scalar ones.
Note that this description applies regardless of how the diagonal blocks are factored.
Thus, this is the core factorization logic of \cref{alg:beam} that is
applied to the modified matrix \(\widetilde{A}\), with
line~\ref{alg:block_lu:diag} of \cref{alg:block-lu} corresponding to
computing the \ac{svd} of \(A^{(k)}_{k,k} = U_k \Sigma_k V_k^T\) and
setting \(L_{k,k} = U_k\) and \(R_{k,k} = \Sigma_k V_k^T\).
Alternatively, the previous analysis by Demmel et al.~\cite{demmelStabilityBlockLU1995} takes \(L_{k,k}=I\) and \(R_{k,k} = A^{(k)}_{k,k}\).
Finally, partitioned, non-pivoted LU is equivalent to factoring the diagonal blocks with a pointwise LU and using those factors as the diagonal blocks of \(L\) and \(R\).

Many basic properties of pointwise LU have blockwise analogues.
Most notable is the analogue to strongly nonsingular matrices; a matrix is \textit{block strongly nonsingular} if each of it's leading principal submatrices aligned to the block partitioning are nonsingular.
This definition is critical as \cref{alg:block-lu} completes if and only if \(A\) is block strongly nonsingular.
Easily following from the definition, any strongly nonsingular matrix is also block strongly nonsingular; this is reflected in the fact that cache blocking for pointwise, non-pivoted LU is normally implemented as \cref{alg:block-lu} where \cref{alg:block_lu:diag} is an unblocked, non-pivoted LU routine.

\subsection{Growth Factors}
\label{sec:block-lu:growth}

Because element growth plays a significant role in the backwards error of block LU factorization,
we begin by proposing an extension of the traditional measure of growth that is better suited for blockwise elimination.
For any matrix norm, \(\norm{\alpha}{\cdot}\), we define a growth factor
maximized across the successive Schur complements \(A^{(k)}\) and
parameterized by both the matrix norm \(\alpha\) and the blocking scheme
\(\mathcal{I}\):
\begin{equation}
\label{eq:new-growth-factor}
\blockgrowth[\mathcal{I}]{\alpha} \stackrel{\mbox{\scriptsize def}}{=}
\frac{\max_{1 \leq k \leq |\mathcal{I}|}\norm{\alpha}{A^{(k)}}}{\norm{\alpha}{A}}.
\end{equation}
For notational simplicity, we omit the blocking factor when it is obvious from the context and we denote the growth for a constant block size, \(n_b\), as \(\uniblockgrowth[n_b]{\alpha}\).
Interestingly, the growth is independent of how the diagonal blocks are factored (in exact arithmetic) since it depends on only the Schur complements.
Note that Wilkinson's classic growth factor~\cite{wilkinsonRoundingErrorsAlgebraic1963} equals \(\uniblockgrowth{\maxrm}\),
that Amodio and Mazzia's growth factor~\cite{amodioNewApproachBackward1999} equals \(\uniblockgrowth{\infty}\),
and that George and Ikramov's growth factor for a subordinate matrix
norm \(\norm{\alpha}{\cdot}\)~\cite{georgeBlockLUFactorization2005}
equals \(\blockgrowth{\max\alpha}\).
Barlow and Zha's growth factor~\cite{barlowGrowthGaussianElimination1998} is related to, but strictly greater than, \(\uniblockgrowth{2}\);
the difference arises from their inclusion of elements from \(U\) in the numerator's norms.
Similarly, Yeung and Chen's growth factor for  \(U\)~\cite{yeungProbabilisticAnalysisGaussian1997} is less than or equal to \(\uniblockgrowth{\infty}\).

The key benefit of this generalized definition is that it measures the growth that is actually incurred during factorization.
For example, consider the matrix
\[
	A = \begin{bmatrix}
		a & 1 & 0 & 0 \\
		1 & 0 & 0 & 0 \\
		0 & 0 & 1 & 0 \\
		0 & 0 & 0 & 1
	\end{bmatrix}
\]
parameterized by \(a \ll 1\) when factored by block LU with a block size of 2 (i.e., \(\mathcal{I} = [1, 3, 5]\)).
Obviously, such a factorization will be very stable regardless of how small \(a\) gets.
However, the Wilkinson growth factor of non-pivoted, pointwise LU will go to infinity as \(a\) approaches 0; thus, the analysis of Demmel et al. cannot meaningfully bound the stability~\cite{demmelStabilityBlockLU1995}.
In contrast, the blockwise growth factor in \cref{eq:new-growth-factor} is 1 (since the only Schur-complement is a submatrix of the original matrix) which better captures the stability of the factorization.

Because \cref{eq:new-growth-factor} is parameterized for both the blocking and the norm, we provide the following theorem relating different versions of the measure.
Importantly, this theorem allows us to bound the new growth factors in terms of Wilkinson's classic growth factor:
\begin{align}
	\label{eq:new-growth-vs-wilkinson}
	\blockgrowth{\alpha} &\leq \uniblockgrowth{\alpha} \leq n\uniblockgrowth{\maxrm}
	&
	\alpha \in \{1,2,\infty,F\}.
\end{align}
(Recall that the index set denoted by \(\langle 1\rangle\) is the integers \([1, 2, \dots, n+1]\).)

\begin{theorem}
	\label{thm:growth-relations}
	Let \(\norm{\alpha}{\cdot}\) and \(\norm{\beta}{\cdot}\) be matrix norms.
	Then, the following relations between growth factors hold:
	\begin{enumerate}
		\item If \(\mathcal{J}\subseteq\mathcal{I}\), then \(\blockgrowth[\mathcal{J}]{\alpha} \leq \blockgrowth[\mathcal{I}]{\alpha}\).
		\item If \(\mu^{-1}\norm{\alpha}{A} \leq \norm{\beta}{A} \leq \nu\norm{\alpha}{A}\) for any square matrix \(A\),
			then \[
				(\mu\nu)^{-1}\blockgrowth[\mathcal{I}]{\alpha} \leq \blockgrowth[\mathcal{I}]{\beta} \leq \mu\nu\blockgrowth[\mathcal{I}]{\alpha}.
			\]
	\end{enumerate}
\end{theorem}
\begin{proof}
	The first result follows from the observation that, by the quotient formula for Schur complements~\cite[Thm.~1.4]{zhangSchurComplementIts2005}, all indices in \(\mathcal{J}\) also occur in  \(\mathcal{I}\), so all of the trailing matrices produced by the former must also be produced by the latter.
	In other words,
	\begin{align*}
		\blockgrowth[\mathcal{J}]{\alpha}
		&= \frac{\max_{j \in  \mathcal{J}}\norm{\alpha}{A/A[1:j\mathord{-}1,1:j\mathord{-}1]}}{\norm{\alpha}{A}}
		\leq \frac{\max_{i \in \mathcal{I}}\norm{\alpha}{A/A[1:i\mathord{-}1,1:i\mathord{-}1]}}{\norm{\alpha}{A}}
		= \blockgrowth[\mathcal{I}]{\alpha}.
	\end{align*}
	where \(A/A[1:k\mathord{-}1,1:k\mathord{-}1]\) is the Schur complement for the \(A[1:k\mathord{-}1,1:k\mathord{-}1]\) leading block.

The second result is a straightforward substitution of norm
equivalencies, some of which were alluded to in
\cref{sec:intro-norms}.
\end{proof}

However, \cref{eq:new-growth-vs-wilkinson} may be quite pessimistic for some matrices.
So, it is also useful to directly study the growth of \cref{eq:new-growth-factor}.
Applying the triangle inequality to the Schur-complement provides the following theorem which implies that, like the pointwise case, growth is caused by leading-principal submatrices with large inverses.
\begin{theorem}
	\label{thm:schur-growth-bound}
	Let \(\norm{\alpha}{\cdot}\) be a submultiplicative matrix norm.
	Then,
	\begin{align}
		\nonumber
		\blockgrowth{\alpha}
		&\leq 1+\max_{1\leq k \leq n_t-1} \left[ \norm{\alpha}{A_{1:k,1:k}^{-1}} \min\left(\norm{\alpha}{A_{k+1:n_t,1:k}}, \norm{\alpha}{A_{1:k,k+1:n_t}}\right) \right] \\
		&\leq 1+\max_{1\leq k \leq n_t-1}\, \norm{\alpha}{A_{1:k,1:k}^{-1}}\norm{\alpha}{A}.
		\label{eq:growth-bound-leading-submatrices}
	\end{align}
\end{theorem}

Note that a similar bound exists for the max-norm, except with an extra factor of \(k\) to account for the lack of submultiplicativity.
These types of bounds provide a valuable heuristic into the behavior of the growth factor, but have received limited attention in the literature to date.

\subsection{Backward Error Bounds}
\label{sec:block-lu:blockwise}

We begin by analyzing the backward error of the factorization itself.
For the sake of flexibility, we parameterize the error of the blockwise operations.
When the individual block operations are numerically stable, the resulting coefficients, \(c_{11}\) et al., will generally be moderate polynomials
in \(n\) or \(n_b\), often denoted as either \(p(n)\) or \(p(n_b)\),
respectively, 
and as such, will generally not be considered in further detail.
Note that this theorem specifically applies to only implementations where the Schur-complements are performed in order, unlike some pointwise analyses.

\begin{theorem}
	\label{thm:block-lu-submult}
	For a given norm, \(\norm{\alpha}{\cdot}\), 
	apply \cref{alg:block-lu} such that for all \(1\leq k\leq n_t\),
	\begin{align}
		\label{eq:block-lu:assumptions:11}
		A_{kk}^{(k)} & = \widehat{L}_{kk}\widehat{R}_{kk} + E_{11}^{(k)}, & \norm{\alpha}{E_{11}^{(k)}} &\leq c_{11}^{(k)}u\norm{\alpha}{A_{kk}^{(k)}}, \\
		\label{eq:block-lu:assumptions:12}
		A_{k\mathcal{T}_k}^{(k)} &= \widehat{L}_{kk}\widehat{R}_{k\mathcal{T}_k} + E_{12}^{(k)}, & \norm{\alpha}{E_{12}^{(k)}} &\leq c_{12}^{(k)}u\norm{\alpha}{\widehat{L}_{kk}}\norm{\alpha}{\widehat{R}_{k\mathcal{T}_k}}, \\
		\label{eq:block-lu:assumptions:21}
		A_{\mathcal{T}_kk}^{(k)} &= \widehat{L}_{\mathcal{T}_kk}\widehat{R}_{kk} + E_{21}^{(k)}, & \norm{\alpha}{E_{21}^{(k)}} &\leq c_{21}^{(k)}u\norm{\alpha}{\widehat{L}_{\mathcal{T}_kk}}\norm{\alpha}{\widehat{R}_{kk}}, \\
		\label{eq:block-lu:assumptions:22}
		A_{\mathcal{T}_k\mathcal{T}_k}^{(k+1)} &= A_{\mathcal{T}_k\mathcal{T}_k}^{(k)} - \widehat{L}_{\mathcal{T}_kk}\widehat{R}_{k\mathcal{T}_k} + E_{22}^{(k)}, & \norm{\alpha}{E_{22}^{(k)}} &\leq c_{22A}^{(k)}u\norm{\alpha}{A_{\mathcal{T}_k\mathcal{T}_k}^{(k)}} \\
		& & & \phantom{\leq} + c_{22LU}^{(k)}u\norm{\alpha}{\widehat{L}_{\mathcal{T}_kk}}\norm{\alpha}{\widehat{R}_{k\mathcal{T}_k}},\nonumber
	\end{align}
	where \(\mathcal{T}_k = k+1:n_t\) are the trailing blocks.
	Then, the computed factorization satisfies
	\[
		\norm{\alpha}{A - \widehat{L}\widehat{R}}
		\leq C_{A}u\blockgrowth{\alpha}\norm{\alpha}{A} + C_{LU}u\norm{\alpha}{\widehat{L}}\norm{\alpha}{\widehat{R}}
	\]
	where the \(C_{*}\) constants are bound by the sums of the
	error-term coefficients of the individual block operations:
	\[
		\label{eq:block-lu:coeff-bounds}
		C_A \leq \sum_{k=1}^{n_t}c_{11}^{(k)} + \sum_{k=1}^{n_t-1}c_{22A}^{(k)}
		\qquad\text{and}\qquad
		C_{LU} \leq \sum_{k=1}^{n_t-1} \left(c_{21}^{(k)} + c_{12}^{(k)} + c_{22LU}^{(k)}\right).
	\]
\end{theorem}
\begin{proof}
	When \(A\) is a single block, \cref{eq:block-lu:assumptions:11} and \(\blockgrowth{\alpha} \geq 1\) gives \(\norm{\alpha}{E} \leq c_{11}u\blockgrowth{\alpha}\norm{\alpha}{A}\).
	So, assume there are multiple blocks.
	Writing out the blocks of the product \(A - \widehat{L} \widehat{R}\) and substituting \cref{eq:block-lu:assumptions:11,eq:block-lu:assumptions:12,eq:block-lu:assumptions:21,eq:block-lu:assumptions:22} gives
	\begin{align*}
		\norm{\alpha}{A - \widehat{L}\widehat{R}}
		&= \norm*{\alpha}{\left[\begin{array}{ll}
			A^{(1)}_{1,1} \hspace{0.45em} - L_{1,1} R_{1,1}
			& A^{(1)}_{1,\mathcal{T}_1} \hspace{0.45em} - L_{1,1} R_{1,\mathcal{T}_1} \\
			A^{(1)}_{\mathcal{T}_1,1} - L_{\mathcal{T}_1,1}R_{1,1}
			& A^{(1)}_{\mathcal{T}_1,\mathcal{T}_1} - L_{\mathcal{T}_1,1}R_{1,\mathcal{T}_1} - R_{\mathcal{T}_1,\mathcal{T}_1}L_{\mathcal{T}_1,\mathcal{T}_1}
		\end{array}\right]} \\
		&\leq \norm{\alpha}{E_{11}^{(1)}}+\norm{\alpha}{E_{12}^{(1)}}+\norm{\alpha}{E_{21}^{(1)}}+\norm{\alpha}{E_{22}^{(1)}} + \norm{\alpha}{A^{(2)} - \widehat{L}_{\mathcal{T}_1\mathcal{T}_1}\widehat{R}_{\mathcal{T}_1\mathcal{T}_1}} \\
		&\leq (c_{11}^{(1)} + c_{22A}^{(1)})u\blockgrowth{\alpha}\norm{\alpha}{A}
		+ (c_{12}^{(1)} + c_{21}^{(1)} + c_{22LU}^{(1)})u\norm{\alpha}{\widehat{L}}\norm{\alpha}{\widehat{R}} \\
		&\phantom{\leq} + \norm{\alpha}{A^{(2)} - \widehat{L}_{\mathcal{T}_1\mathcal{T}_1}\widehat{R}_{\mathcal{T}_1\mathcal{T}_1}}.
	\end{align*}
	Continuing the iteration on the trailing matrix term gives
	\begin{align*}
		C_A &\leq \sum_{k=1}^{n_t}c_{11}^{(k)} + \sum_{k=1}^{n_t-1}c_{22A}^{(k)}
		&&\text{and} &
		C_{LU} &\leq \sum_{k=1}^{n_t-1} (c_{12}^{(k)} + c_{21}^{(k)} + c_{22LU}^{(k)}).
	\end{align*}
\end{proof}

For \ac{beam}, the best algorithm to compute the diagonal SVDs is either the QR algorithm or the divide and conquer algorithm, both of which are normwise backward stable~\cite{demmelAccurateSingularValues2006,guDivideandConquerAlgorithmBidiagonal2006}
Furthermore, the remaining block operations are composed out of matrix multiplication and inverting diagonal matrices, both of which are forward stable~\cite{highamAccuracyStabilityNumerical2002}.
Thus, \cref{thm:block-lu-submult} applies to \ac{beam}.
Alternatively, when block LU takes \(\widehat{L}_{kk} = I\) and \(\widehat{R}_{kk} = A_{kk}^{(k)}\) and \cref{eq:block-lu:assumptions:21} is solved using pointwise LU with partial pivoting, \(c_11^{(k)}\) and \(c_{12}^{(k)}\) are trivially zero.
Furthermore, the standard analysis of pointwise LU bounds \(c_{21}^{(k)}\) as the growth factor of \(A^{(k)}_{kk}\) times a polynomial of \(n_b\), and the standard analysis of matrix multiplication bounds the remaining two constants as polynomials of \(n_b\).

Next, we similarly bound the error for block-triangular solves.
\Cref{alg:block-tri-solve} shows the algorithm for a lower block-triangular matrix; the upper triangular case is analogous.
To ensure the bound has the correct structure for any submultiplicative norm, \(\norm{\alpha}{\cdot}\), we introduce the quantity \(\mu_\alpha\).
This quantity is the smallest value such that for any \(\vec{u},\vec{v}\in\mathbb{R}^n\) there is a matrix, \(M\), such that \(Mv = u\) and \(\norm{\alpha}{M} \norm{\alpha}{v} \leq \mu_\alpha\norm{\alpha}{u}\).
This value is trivially 1 for Schatten norms (e.g., the spectral and Frobenius norms).
Furthermore, one can show that this quantity is 1 for the elementwise max norm and for operator norms induced by a \(\ell_p\) vector norm.
This quantity it always at least 1; 
however, for some norms, it can be larger than 1.
For example, the elementwise 1-norm, \(\norm{\mathrm{sum}}{\cdot}\), which equals the sum of element magnitudes and is both absolute and submultiplicative has \(\mu_{\mathrm{sum}} = n\).


\begin{algorithm}[tpb]
\caption{Solve Lower Block Triangular System, \(Ly = b\)}
\fixlabel{alg:block-tri-solve}
\begin{algorithmic}[1]
	\State \(n_t \gets \) the number of blocks in \(\mathcal{I}\)
	\State \(b^{(1)} \gets b\)
	\For{\(k = 1:n_t\)}
		\State \(y_k \gets L_{k,k}^{\,-1} b_k^{(k)}\)
		\State \(b^{(k+1)}_{\range{k+1}{n}} \gets b^{(k)}_{\range{k+1}{n}} - L_{\range{k+1}{n},k} y_k\)
	\EndFor
\end{algorithmic}
\end{algorithm}

\begin{theorem}
	\label{thm:block-lu:forward-sub}
	Let \(\norm{\alpha}{\cdot}\) be a submultiplicative norm and let
	\[
		\mu_\alpha =
		\max_{\substack{u,v\in\mathbb{R}^n\\v\neq 0}}
		\min_{\substack{M\in\mathbb{R}^{n\times n} \\ u = Mv}}
		\frac{\norm{\alpha}{M}\norm{\alpha}{v}}{\norm{\alpha}{Mv}}
	\]
	bound the norm of a mapping between an arbitrary pair of vectors.
	Assume that \(\widehat{L}\) is a lower block-triangular matrix, the right-hand side \(b\) is nonzero, and that
	\begin{alignat}{2}
		\label{eq:block-tri-solve:assumptions:1}
		(\widehat{L}_{kk} + E_{L1}^{(k)})\widehat{y}_{k} &= b^{(k)}_k
		& \norm{\alpha}{E_{L1}^{(k)}} &\leq c_{L1}^{(k)}u\norm{\alpha}{\widehat{L}_{kk}} \\
		\label{eq:block-tri-solve:assumptions:2}
		b^{(k+1)}_{\mathcal{T}_k} + E_{L2}^{(k)} &= b^{(k)}_{\mathcal{T}_k} - \widehat{L}_{\mathcal{T}_kk}\widehat{y}_k
		&\hspace{0.95em} \norm{\alpha}{E_{L2}^{(k)}} &\leq c_{L2}^{(k)}u(\norm{\alpha}{b^{(k)}_{\mathcal{T}_k}} + \norm{\alpha}{\widehat{L}_{\mathcal{T}_kk}}\norm{\alpha}{\widehat{y}_k})
	\end{alignat}
	with \(\mathcal{T}_k=k+1:n_t\) being the indices of the trailing blocks.
	Then, for the computed solution, \(\widehat{y}\), there exists a matrix, \(E\), such that
	\[
		(\widehat{L} + E)\widehat{y} = b
		\qquad\text{where}\qquad
		\norm{\alpha}{E} \leq \mu_\alpha\left(\sum_{k=1}^{n_t}c_{L1}^{(k)} + 2\sum_{k=1}^{n_t-1}c_{L2}^{(k)}\right)u \norm{\alpha}{\widehat{L}} + \Oh{u^2}.
	\]
	An analogous bound holds for upper block-triangular matrices.
	Furthermore, for a non-submultiplicative norm, a similar bound holds, except with an additional factor of \(n\).
\end{theorem}
\begin{proof}
Writing out the blocks of the product \(b - \widehat{L} \widehat{y}\) gives
\begin{align}
	\norm{\alpha}{b - \widehat{L}\widehat{y}}
	\nonumber &\leq \norm{\alpha}{E^{(1)}_{1}}\norm{\alpha}{\widehat{y}_1} + \norm{\alpha}{E_{L2}^{(1)}} + \norm{\alpha}{b^{(2)}_{\mathcal{T}_1}-\widehat{L}_{\mathcal{T}_1\mathcal{T}_1}\widehat{y}_{\mathcal{T}_1}}
	\\
	\label{eq:block-tri-solve:expanded}
	&\leq c_{L1}^{(k)}u\norm{\alpha}{\widehat{L}_{11}}\norm{\alpha}{\widehat{y}_1} + c_{L2}^{(k)}u(\norm{\alpha}{b^{(2)}_{\mathcal{T}_1}} + \norm{\alpha}{\widehat{L}_{\mathcal{T}_11}}\norm{\alpha}{\widehat{y}_k}) \\
	\nonumber &\phantom{\leq=}+ \norm{\alpha}{b^{(2)}_{\mathcal{T}_1}-\widehat{L}_{\mathcal{T}_1\mathcal{T}_1}\widehat{y}_{\mathcal{T}_1}}.
\end{align}
Using
\(
	\norm{\alpha}{b^{(2)}_{\mathcal{T}_1}}
	\leq \norm{\alpha}{b^{(2)}_{\mathcal{T}_1}-\widehat{L}_{\mathcal{T}_1\mathcal{T}_1}\widehat{y}_{\mathcal{T}_1}} + \norm{\alpha}{\widehat{L}_{\mathcal{T}_1\mathcal{T}_1}\widehat{y}_{\mathcal{T}_1}}
\),
we can simplify \cref{eq:block-tri-solve:expanded} to
\[
	\norm{\alpha}{b - \widehat{L}\widehat{y}}
	\leq (c_{L1}^{(k)} + 2c_{L2}^{(k)})u\norm{\alpha}{\widehat{L}}\norm{\alpha}{\widehat{y}} +  (1+c_{L2}^{(k)}u)\norm{\alpha}{b^{(2)}_{\mathcal{T}_1}-\widehat{L}_{22}\widehat{y}_{\mathcal{T}_1}}.
\]
Continuing the iteration gives
\[
	\norm{\alpha}{b - \widehat{L}\widehat{y}}
	\leq \left(\sum_{k=1}^{n_t}c_{L1}^{(k)} + 2\sum_{k=1}^{n_t-1}c_{L2}^{(k)}\right)u \norm{\alpha}{\widehat{L}}\norm{\alpha}{\widehat{y}} + \Oh{u^2}.
\]
By the definition of \(\mu_\alpha\),
there exists an error matrix \(E\) such that
\(E\widehat{y} = b - \widehat{L}\widehat{y}\)
and
\(\norm{\alpha}{E}\norm{\alpha}{\widehat{y}} \leq \mu_\alpha\norm{\alpha}{b - \widehat{L}\widehat{y}}\).
Therefore, \((\widehat{L} + E)\widehat{y} = b\) and
\begin{align*}
	\norm{\alpha}{E} &\leq \mu_\alpha\left(\sum_{k=1}^{n_t-1}c_{L1}^{(k)} + 2\sum_{k=1}^{n_t}c_{L2}^{(k)}\right)u \norm{\alpha}{\widehat{L}} + \Oh{u^2}.
\end{align*}

Note that the separation of \(\norm{\alpha}{E_1^{(1)}\widehat{y}}\) and of \(\norm{\alpha}{\widehat{L}_{\mathcal{T}_1\mathcal{T}_1}\widehat{y}_{\mathcal{T}_1}}\) in \cref{eq:block-tri-solve:expanded} and the subsequent simplification are the only places where submultiplicativity is used.
Thus, for a non-submultiplicative norm, there is an extra factor of \(n\) for the \(c_{L1}^{(k)}\) and \(c_{L2}^{(k)}\) terms, respectively, resulting in an overall increase by a factor of \(n\).
\end{proof}

As was the case with \cref{thm:block-lu-submult}, standard analyses of diagonal matrices and matrix multiplication show that the constants are polynomials of \(n_b\) for BEAM, and standard analyses of matrix multiplication and LU show that the constants of polynomials of \(n_b\) times the per-block growth factor for block LU using pointwise LU for the diagonal blocks~\cite{highamAccuracyStabilityNumerical2002}.

\Cref{thm:block-lu-submult,thm:block-lu:forward-sub} can be combined to bound the backward error of a full, blockwise solve.
\begin{corollary}
	\label{thm:block-lu-solve}
	Let \(\norm{\alpha}{\cdot}\) be a submultiplicative norm.
	Solve \(Ax=b\) using the block LU and solve of \cref{alg:block-lu,alg:block-tri-solve} under the assumptions of \cref{thm:block-lu-submult} and \cref{thm:block-lu:forward-sub}.
	Then, there exists a matrix \(E\) such that the computed solution, \(\widehat{x}\) satisfies \((A+E)\widehat{x} = b\) where
	\begin{align*}
		\norm{\alpha}{E} &\leq C_{A}u\blockgrowth{\alpha}\norm{\alpha}{A} + C_{LU}u\norm{\alpha}{\widehat{L}}\norm{\alpha}{\widehat{R}} + \Oh{u^2}, \\
		C_A &\leq \sum_{k=1}^{n_t}c_{11}^{(k)} + \sum_{k=1}^{n_t-1}c_{22A}^{(k)},
		\hspace{18.5em}\text{and} \\
		C_{LU} &\leq \mu_\alpha\sum_{k=1}^{n_t}\left[c_{L1}^{(k)} + c_{R1}^{(k)}\right]
		+ \sum_{k=1}^{n_t-1}\left[c_{21}^{(k)} + c_{12}^{(k)} + c_{22LU}^{(k)} + 2\mu_\alpha\left(c_{L2}^{(k)} + c_{R2}^{(k)}\right)\right].
	\end{align*}
	As per \cref{thm:block-lu:forward-sub}, if the norm is instead not submultiplicative, the \(\mu_\alpha\) terms increase by a factor of \(n\).
\end{corollary}
\begin{proof}
	The combination of algorithms computes \(\widehat{L}, \widehat{R}, \widehat{y}, \widehat{x}\) such that
	\begin{align*}
		A+E_A &= \widehat{L}\widehat{R}, &
		(\widehat{L} + E_L)\widetilde{y} &= b, &
		&\text{and} &
		(\widehat{R} + E_R)\widetilde{x} &= \widetilde{y}
	\end{align*}
	where \(E_A\), \(E_L\), and \(E_R\) come from \cref{thm:block-lu-solve}, \cref{thm:block-lu:forward-sub}, and the upper-triangular analogue of \cref{thm:block-lu:forward-sub}, respectively.
	Combining these equations gives
	\[
		E = - E_A + E_L\widehat{R} + \widehat{L}E_R + E_LE_R
	\]
	from which the norm bound can easily be computed.
\end{proof}

These theorems raise the question whether \(\norm{\alpha}{\widehat{L}}\norm{\alpha}{\widehat{R}}\) can be bounded in terms of \(\norm{\alpha}{A}\).
We focus on bounding the norms of the exactly-computed factors \(L\) and \(R\) since, even if \(\widehat{L}\) and \(\widehat{R}\) only have 1 bit of accuracy, they still satisfy \(\norm{\alpha}{\widehat{L}}\norm{\alpha}{\widehat{R}} < 4\norm{\alpha}{L}\norm{\alpha}{R}\).
Meaningfully bounding this quantity in complete generality is difficult,
so we focus on the case where a (unitarily-invariant) Schatten norm is used and the diagonal
blocks of \(L\) are unitary, as is the case for \ac{beam}.
Analogous bounds exist whenever the norm is invariant to the diagonal blocks of \(L\), for example the 1 or \(\infty\) operator norms with permutation matrices.
Note that, in addition to expanding its scope, this bound improves the general bound of Demmel et al.~\cite[pg.~182]{demmelStabilityBlockLU1995}, replacing the pointwise growth factor cubed with the blockwise growth factor squared.

\begin{theorem}
	\label{thm:block-lu:LR-bounds}
	Let \(A=LR\) be a block LU factorization where the diagonal blocks of \(L\) are unitary, and let \(\norm{p}{\cdot}\) be a Schatten \(p\)-norm.
	Then
	\[
		\norm{p}{L} \leq n_b^{1/p}n_t + n_t\blockgrowth{p}\cond{p}{A}
		\qquad\text{and}\qquad
		\norm{p}{R} \leq n_t\blockgrowth{p}\norm{p}{A}
	\]
	where \(n_b\) is the largest block size and for the spectral norm \(n_b^{1/p} = 1\).
	Thus,
	\[
		\norm{p}{L}\norm{p}{R} \leq n_t^2\blockgrowth[\,2]{p}\cond{p}{A}\norm{p}{A} + n_b^{1/p}n_t^2\blockgrowth{p}\norm{p}{A}.
	\]
	By symmetry, an analogous result holds if the diagonal blocks of \(R\) are unitary instead.
\end{theorem}
\begin{proof}
First, consider the \(k\)th block row of \(R\), which \cref{alg:block-lu} computes as
\(R_{k,k:n_t} = L_{kk}^{-1}A_{k,k:n_t}^{(k)}\).
Combining this with the definition of the growth factor gives
\[
	\norm{p}{R_{k,k:n_t}}
	= \norm{p}{L_{kk}^{-1}A_{k,k:n_t}^{(k)}}
	= \norm{p}{A_{k,k:n_t}^{(k)}}
	\leq \blockgrowth{p}\norm{p}{A}
\]
Then,
\[
	\norm{p}{R}
	\leq \sum_{k=1}^{n_t} \norm{p}{R_{k,k:n_t}}
	\leq n_t\blockgrowth{p}\norm{p}{A}.
\]
Similarly, for the \(k\)th column of \(L\),
\begin{equation}
	\label{eq:LR-bounds:L-col}
	\norm{p}{L_{k+1:n_t,k}}
	= \norm{p}{A_{k+1:n_t,k}^{(k)}R_{kk}^{-1}}
	= \norm{p}{A_{k+1:n_t,k}^{(k)}R_{kk}^{-1}L_{kk}^{-1}}
	= \norm{p}{A_{k+1:n_t,k}^{(k)}(A_{kk}^{(k)})^{-1}}.
\end{equation}
Blockwise inversion shows that \(-(A^{(k+1)})^{-1}A^{(k)}_{k+1:n_t,k}(A^{(k)}_{kk})^{-1} = ((A^{(k)})^{-1})_{k+1:n_t,k}\) \cite[p.~182]{demmelStabilityBlockLU1995}.
Thus,
\begin{align*}
	\norm{p}{L}
	\leq \sum_{k=1}^{n_t} \norm{p}{L_{k,k}} + \sum_{k=1}^{n_t} \norm{p}{L_{k+1:n_t,k}}
	\leq n_b^{1/p}n_t + n_t\blockgrowth{p}\cond{p}{A}.
\end{align*}
\end{proof}

Note that \(\norm{}{A_{k+1:n_t,k}^{(k)}(A^{(k)}_{kk})^{-1}}\) is the key quantity to bound \(\norm{p}{L}\).
However, \ac{beam} is designed specifically to prevent
\(\left(A^{(k)}_{kk}\right)^{-1}\)
from growing catastrophically large, which provides for an additional bound on \(\norm{p}{L}\) for ill-conditioned matrices.

\begin{theorem}
	Let \(A=LR\) be a block LU factorization where the diagonal blocks of \(L\) are unitary.
	Assume that the smallest singular value of \(A^{(k)}_{kk}\) is at least \(\hat{\tau}\norm{2}{A}\) (as is ensured by \ac{beam}).
	Then,
	\[
		\norm{2}{L} \leq n_t + n_t\blockgrowth{2}\hat{\tau}^{-1}.
	\]
\end{theorem}
\begin{proof}
	By the assumption, \(\norm{2}{(A^{(k)}_{kk})^{-1}} \leq \hat{\tau}^{-1}\norm{2}{A}^{-1}\).
	Hence,
	\[
		\norm{2}{A_{k+1:n_t,k}^{(k)}(A^{(k)}_{kk})^{-1}} \leq \blockgrowth{2}\hat{\tau}^{-1}.
	\]
	From there, the proof follows that of \cref{thm:block-lu:LR-bounds}.
\end{proof}


\subsection{Diagonal Dominance}
\label{sec:block-lu:diag-dom}
There are several key classes of matrices that can be accurately factored using pointwise LU without pivoting.
Ideally, block LU, and specifically \ac{beam}, can also accurately factor these matrices, to allow block LU to replace the pointwise version without degrading accuracy.

For pointwise LU, pointwise diagonally dominant matrices are a key class of problems which can be accurately solved without pivoting.
As mentioned in \cref{sec:block-lu:growth}, \(A^{(k)}\) is independent of how the diagonal blocks are factored.
Thus, combining the growth of pointwise LU with \cref{thm:growth-relations} gives
\begin{align*}
	\blockgrowth[\mathcal{I}]{\infty} &= 1 & &\text{when strictly diagonally dominant by rows~\cite[Thm. 5.2]{amodioNewApproachBackward1999},} \\
	\blockgrowth[\mathcal{I}]{1} &= 1 & &\text{when strictly diagonally dominant by columns (analogously), and} \\
	\blockgrowth[\mathcal{I}]{\maxrm} &\leq 2 & &\text{when strictly diagonally dominant by either~\cite{wilkinsonErrorAnalysisDirect1961}}
\end{align*}
for any blocking \(\mathcal{I}\).
Similar analysis can show that block LU has small growth for any matrix where pointwise, non-pivoted LU has small growth, such as symmetric positive-definite matrices, totally non-negative matrices, and generalized Higham matrices~\cite{georgeGrowthFactorGaussian2002}.
Furthermore, for matrices that are pointwise diagonally dominant by
columns or symmetric positive definite, we can improve the bound on the
L-factor of the LU factorization \(\norm{}{L}\) from
\cref{thm:block-lu:LR-bounds}.
\begin{theorem}
	Let \(A=LR\) be a block LU factorization where the diagonal blocks of \(L\) are unitary.
	If \(A\) is diagonally dominant by columns,
	\[
          \norm{2}{L} \leq (n_b^{3/2} + 1) n_t.
	\]
	If \(A\) is symmetric positive definite,
	\[
          \norm{2}{L} \leq (\cond{2}{A}^{1/2} + 1) n_t.
	\]
\end{theorem}
\begin{proof}
	The proof for diagonally dominant matrices expands on that of Demmel et al.~\cite[p.~182]{demmelStabilityBlockLU1995}.
	Assume \(A\) is diagonally dominant.
	By \cref{eq:LR-bounds:L-col} from \cref{thm:block-lu:LR-bounds}, \(\norm{p}{A_{k+1:n_t,k}^{(k)}(A_{kk}^{(k)})^{-1}}\) is the key quantity to bound.
	As per Demmel et al., consider the non-pivoted, pointwise LU factorization \(A^{(k)} = \overline{L}\overline{U}\).
	(For consistency with the block indexing of \(A^{(k)}\), the
        blocks of \(\overline{L}\) and \(\overline{U}\) are also indexed
        starting with \(k\).)
	Thus,
	\[
		A_{k+1:n_t,k}^{(k)}(A_{kk}^{(k)})^{-1}
		= A_{k+1:n_t,k}^{(k)}\overline{U}_{k,k}^{-1}\overline{L}_{k,k}^{-1}
		= \overline{L}_{k+1:n_t,k}\overline{L}_{k,k}^{-1}.
	\]
	Because \(A\) is diagonally dominant by columns, \[
		\norm{2}{\overline{L}_{k+1:n_t,k}} \leq n_b^{1/2}\norm{1}{\overline{L}_{k+1:n_t,k}} \leq n_b^{1/2},
	\]
	and \(\overline{L}_{k,k}\) is also diagonally dominant by columns.
	Recall that the subdiagonal part of the \(j\)th column of \(\overline{L}_{k,k}^{-1}\) is computed by the iteration:
	\begin{align}
		l_j^{(0)} &= -\overline{L}_{k,k}^{-1}[j+1:n_b, j], \\
		l_j^{(i)}[j+1:i] &= l_j^{(i-1)}[j+1:i], \\
		l_j^{(i)}[i+1:n_b] &= l_j^{(i-1)}[i+1:n_b] - l_j^{(i-1)}[i] \cdot \overline{L}_{k,k}[i+1:n_b, i],
	\end{align}
	with \(l_j^{(n_b)}\) being the resulting column vector.
	Note that \(\norm{1}{\overline{L}_{k,k}[i+1:n_b, i]} \leq 1\).
	So, \begin{align*}
		\norm{1}{l_j^{(i)}[i+1:n_b]}
		&\leq \norm{1}{l_j^{(i-1)}[i+1:n_b]} + |l_j^{(i-1)}[i]| \norm{1}{\overline{L}_{k,k}^{-1}[i+1:n_b, i]} \\
		&\leq \norm{1}{l_j^{(i-1)}[i:n_b]}.
	\end{align*}
	Thus, the subdiagonal elements of \(\overline{L}^{-1}_{kk}\) are each at most 1 in absolute value.
	Let \(T\) be the triangular matrix of all ones, and note that \(T^{-T}T^{-1}\) is the tridiagonal matrix with 2 on the diagonal and -1 on the off-diagonal~\cite{elkiesAnswer2normUpper2011}.
	Hence,
	\[
		\norm{2}{\overline{L}_{k,k}^{-1}} \leq \norm{2}{T} \leq \frac{1}{2\sin\bigr(\frac{\pi}{4n+2}\bigl)} \leq \frac{2n_b+1}{3} \leq n_b.
	\]
	Therefore, by \cref{eq:LR-bounds:L-col},
	\[
		\norm{2}{L_{k+1:n_t,k}}
		\leq \norm{2}{A_{k+1:n_t,k}\overline{U}_{k,k}^{-1}\overline{L}_{k,k}^{-1}}
		\leq n_b^{1/2}\norm{1}{\overline{L}_{k+1:n_t,k}}\norm{2}{\overline{L}_{k,k}^{-1}}
		\leq n_b^{3/2}.
	\]

	The proof for symmetric positive definite matrices follows easily from Lemma~4.1 of Demmel et al.~\cite{demmelStabilityBlockAlgorithms1992}.
\end{proof}

For blocked factorizations, it is also interesting to consider the related \textit{block diagonally dominant matrices}.
Given a matrix norm \(\norm{\alpha}{\cdot}\), a matrix is strictly block diagonally dominant by columns if for all \(1 \leq j \leq n_t\),
\[
	\sum_{\substack{i=1\\i\neq j}}^{n_t}\norm{\alpha}{A_{i,j}} < \norm{\alpha}{(A_{j,j})^{-1}}^{-1},
\]
and analogously for block diagonally dominance by rows.
We assume the blocks for diagonal dominance matches those of the elimination.
Like it's pointwise equivalent, strict block diagonal dominance implies that the matrix is block strongly nonsingular.

The growth factors for the block-norms defined in \cref{sec:intro-norms} (with submultiplicative inner norms) are easily bounded for strictly block-diagonally dominant matrices.
The analysis of block LU by Demmel et al.\ bounds
\(\blockgrowth[\mathcal{I}]{\max\alpha}\!\leq 2\)~\cite[Lemma 3.2]{demmelStabilityBlockLU1995}, and a similar proof shows that \(\blockgrowth[\mathcal{I}]{\Sigma\alpha}\!= 1\) for any submultiplicative norm \(\norm{\alpha}{\cdot}\).
Thus, by the growth-factor equality of \cref{thm:growth-relations}, a block-diagonally dominant matrix has \(\blockgrowth{\alpha} \leq n_t^2\) for any submultiplicative norm and for the Frobenius norm, \(\blockgrowth{F} \leq n_t\).
Furthermore, for matrices block diagonally dominant by columns,
\(\norm{\max1}{A} \leq 2\norm{1}{A}\), and thus, \(\blockgrowth{1} \leq 4\).
Likewise, \(\blockgrowth{\infty} \leq 4\) for matrices block diagonally dominant by rows.

Similarly, the growth factors using the block-norms can be bounded for inverses of block-diagonally dominant matrices.
George and Ikramov have previously analyzed block LU (with identity matrices on the diagonal blocks of \(L\)) on inverses of matrices block-diagonally dominant by rows~\cite{georgeGaussianEliminationStable2004}.
For any submultiplicative inner-norm, \(\norm{\alpha}{\cdot}\), they
showed that \(\blockgrowth{\max\alpha} < 2\).
And because the growth factor is independent of how the diagonal blocks are factored, this bound also holds for any implementation of \cref{alg:block-lu}.
Likewise, because the growth factor is invariant to transposition,
\(\blockgrowth{\max\alpha} < 2\) also holds for inverses of matrices block diagonally dominant by rows.

Using block diagonal dominance and the block max norm, we can improve the bound on \(\norm{p}{L}\norm{p}{R}\) from \cref{thm:block-lu:LR-bounds}.
In fact, the assumption of block diagonal dominance can be relaxed to
\[
	\norm[\big]{p}{A^{(k)}_{i,k}\bigl(A^{(k)}_{k,k}\bigr)^{-1}} \leq 1.
\]
However, this relaxed assumption is harder to determine \textit{a priori}.
For a block size of 1, the assumption is equivalent to the pivot selection constraint in partial pivoting.\footnote{%
Interestingly, this suggests a pivoting strategy for block
factorizations that selects the diagonal block in the \(k\)th step to be
a block in the \(k\)th column and on or below the diagonal that minimizes the product of singular values
\[
  \sigma_{n_b} \left(A^{(k)}_{i,k} \right)^{-1}
  \max_{\substack{k \leq h \leq n_t\\h \neq i}}
  \sigma_1 \left(A^{(k)}_{h,k} \right)
\]
for \(k \leq i \leq n_t\).
However, the additional cost of computing singular values across all the
blocks below the diagonal would likely outweigh any benefits of such an
approach.%
}
In addition to matrices that are strictly block diagonally dominant by columns, inverses of matrices block diagonally dominant by rows satisfy this criteria~\cite[Thm.~4]{georgeBlockLUFactorization2005}.

\begin{theorem}
  \label{thm:block-lu:LR-bounds:orth:pp}
  Let \(\norm{p}{\cdot}\) be the Schatten-\(p\) norm and, in the block LU
  factorization \(A=LR\), let the diagonal blocks of \(L\) be unitary.
	Let \(A\) be strongly nonsingular.
	Assume that \(\norm[\big]{p}{A^{(k)}_{i,k}\bigl(A^{(k)}_{k,k}\bigr)^{-1}} \leq 1\) for \(1 \leq k < i \leq n_t\).
	Then,
	\[
          \norm{\max{}p}{L} \leq n_b^{1/p} \qquad \text{and, by norm inequalities,}\qquad  \norm{p}{L} \leq n_b^{1/p}n_t^2.
	\]
	where \(n_b\) is the largest block size.
	Hence,
	\[
          \norm{\max{}p}{L}\norm{\max{}p}{R} \leq
          n_b^{1/p}\blockgrowth{\max p}\norm{\max{}p}{A}
		\qquad \text{and} \qquad
		\norm{p}{L}\norm{p}{R} \leq n_b^{1/p}n_t^3\blockgrowth{p}\norm{p}{A}.
	\]
\end{theorem}
\begin{proof}
	First, similar to the bounds of \(\norm{p}{R}\) from \cref{thm:block-lu:LR-bounds},
	\[
          \norm{\max{}p}{R} \leq \blockgrowth{\max p}\norm{\max{}p}{A}.
	\]
	Next, the definition of the block-max norm gives
	\[
          \norm{\max{}p}{L}
		\leq \max_{1 \leq k \leq i \leq n_t}\!\!\left(\norm{p}{L_{i,k}}\right)
		\leq \max\left(n_b^{1/p}\!, \max_{1 \leq k < i \leq n_t}\!\!\left(\norm{p}{L_{i,k}}\right)\right).
	\]
	Similar to the proof of \cref{thm:block-lu:LR-bounds}, we have for \(i>k\),
	\[
		\norm{p}{L_{i,k}} = \norm{p}{A^{(k)}_{i,k}\bigl(A^{(k)}_{k,k}\bigr)^{-1}} \leq 1.
	\]
        Therefore, \(\norm{\max{}p}{L} \leq n_b^{1/p}\).
%
%
\end{proof}

Combining \cref{thm:block-lu:LR-bounds:orth:pp} with the bounds on growth factor implies that common formulations of block LU have excellent backward stability when \(A\) is block diagonally dominant by columns, specifically
\[
  \norm{\max{}p}{L}\norm{\max{}p}{R} \leq 2n_b^{1/p}\norm{\max{}p}{A}
	\qquad \text{and} \qquad
	\norm{p}{L}\norm{p}{R} \leq 2n_t^4n_b^{1/p}\norm{p}{A}.
\]
(The latter deriving from the former.)
Furthermore, when \(A\) is the inverse of a matrix block diagonally dominant by rows, we have
\[
  \norm{\max{}p}{L}\norm{\max{}p}{R} \leq 4n_b^{1/p}\norm{\max{}p}{A}
	\qquad \text{and} \qquad
	\norm{p}{L}\norm{p}{R} \leq 4n_t^4n_b^{1/p}\norm{p}{A}.
\]

The case of block diagonal dominance by rows is complicated by the potential for the block rows to be poorly scaled relative to each other, which can result in a large value of \(A^{(k)}_{i,k}\bigl(A^{(k)}_{k,k}\bigr)^{-1}\).
One obvious solution is to make the diagonal blocks of \(R\) unitary instead of those of \(L\); however, this requires detecting that the matrix is block diagonally dominant by rows ahead of time.
However, if the block rows are similarly scaled, such matrices will generally be well conditioned~\cite{varahLowerBoundSmallest1975}.

\subsection{Numerical experiments}
\label{sec:block-lu:experiments}

To better understand theoretical bounds on the error, we tested it experimentally.
We specifically use \ac{beam}'s formulation where the diagonal blocks are factored with the \ac{svd}, the left singular values becomes the diagonal block of \(L\) and the singular values times the right singular values becomes the diagonal block of \(R\).
(Additive modifications are not used in this set of experiments but are revisited in the experiments of \cref{sec:svals:experiments}.)
For all experiments, we use a block size of 4, to help emulate the large numbers of blocks present when the algorithm is applied at scale, and double-precision arithmetic.

The experiments were run with Julia 1.12.6 on an Intel Core Ultra 7 255U CPU.
Our code and data are included in the supplementary materials.

To improve the interpretability of the tests, we test the factors of the backward error bound independently.
First, we tested the effect of varying the growth factor.
We generated a matrix with normally-distributed random elements except the leading diagonal block which equaled \(10^{-c} I\), then the entire matrix was scaled to have unit 2-norm.
Then, varying \(c = 1, 2, \dots, 16\) varied the element growth, as suggested by \cref{thm:schur-growth-bound}.
\Cref{fig:growth-berr} shows the resulting relation of that growth on the backward error.
\begin{figure}
	\centering
	\begin{tikzpicture}
\begin{loglogaxis}[
        width=5in, height=3.3in,
        legend pos=south east,
        xlabel={Block growth (\(\blockgrowth{2}\))}, xmin=1, xmax=1e18,
        xtick={1e1, 1e3, 1e5, 1e7, 1e9, 1e11, 1e13, 1e15, 1e17},
        ylabel=Backward error, ymin=1e-17, ymax=10,
        ytick={1e0, 1e-4, 1e-8, 1e-12, 1e-16},
    ]
    \addplot[dotted, black] coordinates {(1.0, 2.220446049250313e-16) (10.0, 2.220446049250313e-15) (100.0, 2.220446049250313e-14) (1000.0, 2.220446049250313e-13) (10000.0, 2.220446049250313e-12) (100000.0, 2.220446049250313e-11) (1.0e6, 2.220446049250313e-10) (1.0e7, 2.220446049250313e-9) (1.0e8, 2.220446049250313e-8) (1.0e9, 2.220446049250313e-7) (1.0e10, 2.220446049250313e-6) (1.0e11, 2.220446049250313e-5) (1.0e12, 0.0002220446049250313) (1.0e13, 0.002220446049250313) (1.0e14, 0.02220446049250313) (1.0e15, 0.2220446049250313) (1.0e16, 2.220446049250313) (1.0e17, 22.20446049250313) (1.0e18, 222.0446049250313) (1.0e19, 2220.446049250313) (1.0e20, 22204.46049250313) };
    \addlegendentry{\(\blockgrowth{2}u\)};
    \addplot[mark=*] coordinates {(36562.703624038164, 5.42684518517045e-13) (43189.41686786481, 8.214562551601265e-13) (11492.990054015083, 2.4932533785262784e-13) (41688.31359109112, 6.766055538978099e-13) (175361.10699289571, 5.009410540553892e-12) (1.7069318501043892e6, 4.4838077757501676e-11) (1.713199914078853e7, 4.2286435400501007e-10) (1.686021063594849e8, 2.539565983956987e-9) (1.7129718678202362e9, 3.864217650935567e-8) (1.6706086467826519e10, 5.703485124312684e-7) (1.6802881634772913e11, 1.3800404272488797e-5) (1.672682741313039e12, 3.621802188189647e-5) (1.730772218706458e13, 0.00031225036997077865) (1.688279276216988e14, 0.0033875412388751996) (1.7192085206536565e15, 0.03244543960325747) (1.6950255186725e16, 0.3244595501897349) (1.6807480587381827e17, 0.4952848490957508) };
    \addlegendentry{error};
\end{loglogaxis}
\end{tikzpicture}
	\vspace{-3em}
	\caption{Effect of varying the growth factor on the backward error.}
	\label{fig:growth-berr}
\end{figure}
Notably, the resulting error is only linear in the growth factor, which is better than is suggested by \cref{thm:block-lu-submult,thm:block-lu:LR-bounds}, and and much better than is suggested by \cite{demmelStabilityBlockLU1995}.
To better understand this potential gap in the theory, we plot the norms of the \(L\) and \(R\) factors.
(Recall that \(A\) is scaled to unit norm, so the norm of \(R\) can be compared directly to the \(n_t\blockgrowth{}\) of \cref{thm:block-lu:LR-bounds}.)
\begin{figure}
	\centering
	\begin{tikzpicture}
\begin{loglogaxis}[
        width=5in, height=3.3in,
        legend pos=south east,
        xlabel={Block growth (\(\blockgrowth{2}\))}, xmin=1, xmax=1e18,
        xtick={1e1, 1e3, 1e5, 1e7, 1e9, 1e11, 1e13, 1e15, 1e17},
        ylabel=Norm, ymin=1, ymax=1e18,
        ytick={1e1, 1e5, 1e9, 1e13, 1e17},
    ]
    \addplot[dotted, black] coordinates {(1.0, 1.0) (10.0, 10.0) (100.0, 100.0) (1000.0, 1000.0) (10000.0, 10000.0) (100000.0, 100000.0) (1.0e6, 1.0e6) (1.0e7, 1.0e7) (1.0e8, 1.0e8) (1.0e9, 1.0e9) (1.0e10, 1.0e10) (1.0e11, 1.0e11) (1.0e12, 1.0e12) (1.0e13, 1.0e13) (1.0e14, 1.0e14) (1.0e15, 1.0e15) (1.0e16, 1.0e16) (1.0e17, 1.0e17) (1.0e18, 1.0e18) (1.0e19, 1.0e19) (1.0e20, 1.0e20) };
    \addlegendentry{\(\blockgrowth{2}\)};
    \addplot[mark=triangle*] coordinates {(36562.703624038164, 7952.3826258807185) (43189.41686786481, 11907.649720041465) (11492.990054015083, 3323.2890810253684) (41688.31359109112, 33005.39746381762) (175361.10699289571, 342599.5664661251) (1.7069318501043892e6, 3.284540774264758e6) (1.713199914078853e7, 3.3520608348187193e7) (1.686021063594849e8, 3.3733144746551347e8) (1.7129718678202362e9, 3.293374541043551e9) (1.6706086467826519e10, 3.279836004154209e10) (1.6802881634772913e11, 3.324637412598362e11) (1.672682741313039e12, 3.362855345923599e12) (1.730772218706458e13, 3.243520345035251e13) (1.688279276216988e14, 3.291481867099283e14) (1.7192085206536565e15, 3.3991272517546105e15) (1.6950255186725e16, 3.310818491720163e16) (1.6807480587381827e17, 3.366986814296002e17) };
    \addlegendentry{\(\norm{2}{L}\)};
    \addplot[mark=square*] coordinates {(36562.703624038164, 2532.772539400237) (43189.41686786481, 6192.212776867186) (11492.990054015083, 1402.315729590337) (41688.31359109112, 2008.4439435172962) (175361.10699289571, 11440.505673919966) (1.7069318501043892e6, 194806.22796421108) (1.713199914078853e7, 1.7673107915497676e6) (1.686021063594849e8, 1.386017190523451e7) (1.7129718678202362e9, 1.9269771478926313e8) (1.6706086467826519e10, 1.18926171054923e9) (1.6802881634772913e11, 1.415578247685941e10) (1.672682741313039e12, 1.2332125524862642e11) (1.730772218706458e13, 1.671090722646392e12) (1.688279276216988e14, 1.2658826926332363e13) (1.7192085206536565e15, 1.8205561075804884e14) (1.6950255186725e16, 1.5761625539566605e15) (1.6807480587381827e17, 1.3150333121214014e16) };
    \addlegendentry{\(\norm{2}{U}\)};
\end{loglogaxis}
\end{tikzpicture}
	\vspace{-3em}
	\caption{Effect of varying the growth factor on the norms of \(L\) and \(R\).}
	\label{fig:growth-LR}
\end{figure}
Note that both \(L\) and \(R\) scale with \(\blockgrowth{}\).
This indicates that \(\norm{}{L}\norm{}{R}\) is quadratic in \(\blockgrowth{}\), indicating the growth-factor terms in \cref{thm:block-lu:LR-bounds} cannot be improved.
However, it does suggest the \(\norm{}{L}\norm{}{R}\) term in the bound of \cref{thm:block-lu-submult} is pessimistic.

We next tested the backward error of the various types of stable matrices discussed in \cref{sec:block-lu:diag-dom}.
All of these matrices are multiplied on one side by the block diagonal matrix
\[
	D_{ii} = 10^{-10(n_t - i)/(n_t - 1)} I,
\]
where \(D_{ii}\) is the \(i\)th block and \(n_t\) is the number of blocks,
or are multiplied on both sides by \(D^{1/2}\) for the symmetric positive definite test.
(The helps isolate the specific matrix property desired by making the matrix ill-conditioned and ensuring the leading principal submatrices are significantly smaller than \(\norm{}{A}\).)
Finally, \(A\) is scaled to have unit 2-norm.
\Cref{tab:stable-berr} shows the errors.
\begin{table}
\caption{Backward error for various forms of stable matrices}
\label{tab:stable-berr}
	\centering
	\begin{tabular}{lS[table-format=1.2,round-mode=places,round-precision=2]S[table-format=1.1e-2,scientific-notation=true,round-mode=places,round-precision=1]S[table-format=1.1e2,scientific-notation=true,round-mode=places,round-precision=1]S[table-format=1.1e2,scientific-notation=true,round-mode=places,round-precision=1]S[table-format=1.2,scientific-notation=false,round-mode=places,round-precision=2]}
    \toprule
    Matrix & {\(\blockgrowth{2}\)} & {Error} & {\(\cond{2}{A}\)} & {\(\norm{2}{L}\)} & {\(\norm{2}{U}\)} \\ 
    \midrule
    Symmetric pos. definite & 1.0 & 8.067571225365707e-21 & 7.758307720545327e15 & 35056.319657679014 & 0.05342820260015487 \\
    point diag. dom. by cols. & 1.0 & 9.408282751371417e-25 & 9.839695107177664e9 & 1.039580127335864 & 1.0000333832956185 \\
    block diag. dom. by cols. & 1.0000000085911849 & 3.6486345828703635e-25 & 3.709392611851628e12 & 1.0292454581648849 & 1.0000000071202186 \\
    inverse of ``\hspace{5.5em}'' & 1.0 & 1.7149444522881605e-19 & 2.346484798900797e13 & 7.795457568305123e7 & 0.9998087933364828 \\
    point diag. dom. by rows & 1.0 & 1.1017825335580497e-17 & 9.776547118078402e9 & 8.129197057636826e7 & 0.9991341107813642 \\
    block diag. dom. by rows & 1.000000037868655 & 6.340440883841404e-19 & 2.381812020676856e12 & 6.4017285802731775e7 & 1.0000000375438276 \\
    inverse of ``\hspace{5.5em}'' & 1.0 & 6.68146479430996e-25 & 1.5895559197351265e12 & 1.029247440735607 & 0.9999849703795469 \\
    \bottomrule
\end{tabular}

\end{table}
For all of these problems, block LU achieved backward errors less than \(u\).
However, problems for which the bound of \cref{sec:block-lu:diag-dom} is still linear in the condition number (i.e., dominance by rows, or inverses of matrices dominant by columns) have errors in the range \(10^{-16}\) to \(10^{-19}\) while problems where the bound is independent of the condition number have errors in the range \(10^{-24}\) to \(10^{-25}\).
Additionally, the symmetric positive definite case, where the bound is the square root of the condition number, sits in between with an error of approximately \(10^{20}\).
However, it is unclear how much of this difference is driven by the condition number bounds and how much is driven by the diagonal scaling artificially inflating the denominator of the error: \(\norm{2}{A}\norm{2}{x}\).

Based on the first experiment, we first review the norms of \(L\) and \(R\) and the condition number of \(A\), which are included in \cref{tab:stable-berr}.
Those norms show that the norm of \(R\) is less than or equal to \(1\), in alignment with the growth factor bounds.
Furthermore, the norm of \(L\) is small when the matrix is pointwise dominant by columns, is blockwise dominant by columns, or has an inverse which is blockwise dominant by rows, in agreement with \cref{sec:block-lu:diag-dom}.
However, the norm of \(L\) is closer to the condition number for the matrices which are pointwise dominant by rows, are blockwise dominant by rows, or have an inverse which is blockwise dominant by columns.
Finally, the norm of \(L\) for the symmetric positive definite matrix is within a few orders of magnitude of the square root of the condition number.
These results suggest that the condition number factor in the bounds on the norm of \(L\) are accurate.

For both sets of experiments, the bounds on the norms of \(L\) and \(R\) appear to capture the behavior of those norms, but the measured backward error is significantly better than \cref{thm:block-lu-submult,thm:block-lu:forward-sub} suggest.
This suggests that it may be possible to improve the theory's dependency on \(\norm{}{L}\norm{}{R}\).

\subsection{Discussion}

By setting the block size to 1, the blockwise theory can be applied to pointwise algorithms.
For example, applying \cref{thm:block-lu-submult,thm:block-lu:LR-bounds:orth:pp} with a block size of 1, the max-norm, and partial pivoting gives
\[
	A = \widehat{L}\widehat{R} + E
		\qquad\text{where}\qquad
	\norm{\maxrm}{E} \leq 3(n-1)u\uniblockgrowth{\maxrm}\norm{\maxrm}{A},
\]
which is within a factor of 1.5 of Wilkinson's classic bound~\cite[Chap 3, (16.13)]{wilkinsonRoundingErrorsAlgebraic1963}.\footnotemark
\footnotetext{Note that the \(\Oh{n}\) term here is smaller than the often cited \(\Oh{n^3}\) term~\cite[Thm.~9.5]{highamAccuracyStabilityNumerical2002}.  However, the latter bound is based on the infinity norm and allowing the summations to occur in any order.}
Furthermore, applying \cref{thm:block-lu:forward-sub} for solving a triangular system to a block size of 1 and any absolute, submultiplicative norm gives
\[
	\norm{}{E} \leq 3nu\norm{}{\widehat{L}} + \Oh{u^2},
\]
which is within a factor of 3 of traditional pointwise analyses~\cite[Thm.~8.5]{highamAccuracyStabilityNumerical2002}.
However, similar to the comparison with Wilkinson's classic bound, applying \cref{thm:block-lu-submult,thm:block-lu:LR-bounds:orth:pp} with the infinity norm gives
\[
	\norm{\infty}{E} \leq (2n^2-n-1)u\uniblockgrowth{\infty}\norm{\infty}{A} + \Oh{u^2},
\]
which is a factor of \(n\) larger than Amodio and Mazzia's bound~\cite{amodioNewApproachBackward1999}.
This gap comes from Amodio and Mazzia's analysis exploiting the component wise structure of the error terms, while the normwise approach to our analysis can only bound a key part of the error in terms of \(\norm{\alpha}{L}\norm{\alpha}{R}\).

For blockwise LU, our analysis improves the state-of-the-art error bound in two ways.
First, it applies to a more general family of block LU factorizations than just those with identity matrices for the diagonal blocks of \(L\).
Additionally, it tightens the bounds for some types of matrices.
\Cref{tab:comparison} compares our bounds with those of Demmel et al.~\cite[Table~4]{demmelStabilityBlockLU1995} and subsequent results~\cite{highamAccuracyStabilityNumerical2002,georgeGaussianEliminationStable2004} for a block LU with identity matrices on the diagonal blocks of \(L\) and normwise stable block operations.
Note that the coefficients for the polynomial in \(n\), \(p(n)\),
are omitted since they are not provided in the analysis of Demmel et al.
\begin{table}
	\caption{Comparison of stability results for block LU, excluding the polynomial-in-\(n\), \(p(n)\), terms.}
	\fixlabel{tab:comparison}
	\centering
	\addtolength{\tabcolsep}{-3pt}
	\begin{tabular}{lccc}
		\toprule
		Matrix & This work & Previous & Source \\
		\midrule
		symmetric pos. definite & \(\cond{}{A}^{1/2}\) & \(\cond{}{A}^{1/2}\) & Demmel et al.~\cite{demmelStabilityBlockLU1995} \\
		point diag. dom. by cols. & 1 & 1 & Higham~\cite{highamAccuracyStabilityNumerical2002}  \\
		block diag. dom. by cols. & 1 & 1 & Demmel et al.~\cite{demmelStabilityBlockLU1995}  \\
		inverse of ``\hspace{5.5em}'' & \(\cond{}{A}\) & \((\uniblockgrowth[1]{})^3\cond{}{A}\) & Demmel et al.~\cite{demmelStabilityBlockLU1995} \\
		point diag. dom. by rows  & \(\cond{}{A}\) & \(\cond{}{A}\) & Demmel et al.~\cite{demmelStabilityBlockLU1995} \\
		block diag. dom. by rows  & \(\cond{}{A}\) & \((\uniblockgrowth[1]{})^3\cond{}{A}\) & Demmel et al.~\cite{demmelStabilityBlockLU1995}  \\
		inverse of ``\hspace{5.5em}'' & 1 & 1 & George et al.~\cite{georgeBlockLUFactorization2005} \\
		\ac{beam} modified & \((\blockgrowth[\mathcal{I}]{})^2\min(\cond{}{A},\hat{\tau}^{-1})\) & \((\uniblockgrowth[1]{})^3\cond{}{A}\) & Demmel et al.~\cite{demmelStabilityBlockLU1995} \\
		general & \((\blockgrowth[\mathcal{I}]{})^2\cond{}{A}\) & \((\uniblockgrowth[1]{})^3\cond{}{A}\) & Demmel et al.~\cite{demmelStabilityBlockLU1995} \\
		\bottomrule
	\end{tabular}
	\addtolength{\tabcolsep}{3pt}
\end{table}
These improved bounds show that block LU is more accurate than previously proven.
Most notably, the previous analysis depends on the pointwise growth factor, and is thus only useful for matrices that can be safely factored by pointwise LU.
By instead using a block growth factor, our analysis can be used for matrices that cannot be safely solved with pointwise LU.
For example, a matrix whose leading block is \(\begin{bsmallmatrix}0 & 1\\1 & 0\end{bsmallmatrix}\) effectively has an infinite pointwise growth factor, but doesn't necessarily pose any issue to block LU.
Furthermore, the bound for the general case is only quadratic in the growth factor instead of cubic, showing that the algorithm is more robust in the face of moderate element growth than previously thought.

While our work improves the bound for blockwise LU, \cref{sec:block-lu:experiments} suggests that there is potential for further improvement.
In addition to the gap between the theoretical and experimental results, note that for the pointwise case, the proof that the backward error is linear in the growth factor exploits the fact that the product can be represented as a sum of outer-products, which commutes with the absolute-value~\cite[Lemma~9.6]{highamAccuracyStabilityNumerical2002}.
This suggests that, if a similar formulation in terms of outer-products can be found for block LU, \cref{thm:block-lu-submult} could be improved.

\section{Singular values of modified matrices}
\label{sec:svals}

A key aspect of the \acf{beam} method is the use of additive modifications to reduce element growth by fixing problematic leading blocks.
However, the modifications raise two important questions:
\begin{itemize}
\item Can they make the matrix ill-conditioned?
\item How do they affect element growth?
\end{itemize}
Note that, similar to \cref{eq:growth-bound-leading-submatrices}, the growth for \ac{beam} is bounded by
\[
	\Rho_\alpha \leq 1 + \max_k\norm{\alpha}{\widetilde{A}^{-1}_{1:k, 1:k}}\norm{\alpha}{A}
\]
where \(\widetilde{A}\) is the modified matrix.
Thus, these two questions are both equivalent to whether the smallest singular value of a modified matrix can be bounded.
Note that, by construction, the modifications prevent both \(\widetilde{A}\) and its leading principal submatrices from being exactly singular.
In addition to \ac{beam}, most modern implementations of sparse direct solvers will modify individual diagonal elements that are too small after pivoting~\cite{liMakingSparseGaussian1998,duffDirectMethodsSparse2017}.
Thus, understanding additive modifications will also provide insight into those solvers.

Recall from \cref{sec:intro-beam} that the modified matrix is \(\widetilde{A} = A + M_U M_\Sigma M_V^T\) where \(M_U\) and \(M_V\) are orthogonal.
Furthermore, \(M_\Sigma\) is a diagonal matrix whose elements are bounded by the user-chosen tolerance \(\tau=\hat{\tau}\norm{2}{A}\).
Given this structure to the modifications, we can use the Woodbury formula both to compute \(A^{-1}\) using \(\widetilde{A}^{-1}\) and to bound \(\widetilde{A}^{-1}\) relative to \(A^{-1}\) (using the Woodbury formula in opposite directions).

Previous work showed analytically that if \(\hat{\tau}\cond{2}{A} \ll 1\), neither the modified matrix nor the Woodbury formula's capacitance matrix would be problematically ill-conditioned~\cite[Thms.~4.1 and~4.3]{lindquistUsingAdditiveModifications2023}.
Unfortunately, for matrices with large condition numbers, that analysis is only usable for very small perturbations.
However, previous experimental results showed BEAM successfully solving ill-conditioned systems with moderately sized tolerances when the Woodbury formula is used~\cite{lindquistUsingAdditiveModifications2023}.
Additionally, as per \cref{tab:comparison}, \ac{beam} improves the upper bound on the numerical error when \(\hat{\tau}\cond{}{A} > 1\), making such configurations desirable.
Ideally, the modified matrix will never have worse conditioning than the original; unfortunately, examples can easily be found where the perturbation reduces the smallest singular value, for example, \(\begin{bsmallmatrix}0 & 1 \\ 1 & 0\end{bsmallmatrix}\) with a block size of 1.

If the smallest singular value of the modified matrix can be bounded, the theoretical analysis of \ac{beam} can be tightened.
First, bounding the smallest singular value directly bounds the condition number of \(\widetilde{A}\), which bounds the forward error of solving the corresponding linear systems.
A controlled forward error is crucial for ensuring that the Woodbury formula~\cite{maNoteStabilityShermanMorrisonWoodbury2025} and iterative refinement~\cite{carsonAcceleratingSolutionLinear2018} can correct the modification.
Furthermore, the Woodbury formula involves inverting the capacitance matrix
\(C = I - M_\Sigma M_V^T\widetilde{A}^{-1}M_U\); a bound of the smallest singular value of \(\widetilde{A}\) allows a tigher bound on its condition number, as shown in the following theorem.
Compare the previous bound~\cite[Thm.~4.3]{lindquistUsingAdditiveModifications2023} which requires \(\hat{\tau} < 1\) in order to obtain
\[
	\cond{2}{C} \leq (1 + \tfrac{\hat{\tau}}{1-\hat{\tau}}\cond{2}{\widetilde{A}})(1 + \hat{\tau}\cond{2}{A}).
\]
(Recall that \(\tau=\hat{\tau}\norm{2}{A}\) is a user chosen tolerance where \(\norm{2}{M_\Sigma}\leq\tau\).)
\begin{theorem}
Let \(\widetilde{A} = A + M_U M_\Sigma M_V^T\) be computed per \cref{alg:beam}.
Assume \(\norm{2}{\widetilde{A}^{-1}} \!\leq\! \psi \norm{2}{A^{-1}}\) for some \(\psi\).
Then, the condition number of the capacitance matrix, \(C\), is bounded by
\[
	\cond{2}{C} \leq (1 + \hat{\tau}\psi\cond{2}{A})(1 + \hat{\tau}\cond{2}{A})
\]
\end{theorem}
\begin{proof}
From the previous bound~\cite[Thm.~4.3]{lindquistUsingAdditiveModifications2023},
\[
	\cond{2}{C} \leq (1 + \tau\norm{2}{\widetilde{A}^{-1}})(1 + \tau\norm{2}{A^{-1}}).
\]
Then, substituting the assumption on \(\norm{2}{\widetilde{A}^{-1}}\) and the definition of \(\tau\) gives
\[
	\cond{2}{C} \leq (1 + \hat{\tau}\psi\cond{2}{A})(1 + \hat{\tau}\cond{2}{A}).
\]
\end{proof}

Unfortunately, bounding \(\psi\) is a difficult question.
Because both \(A\) and \(\widetilde{A}\) are nonsingular, applying the Woodbury formula~\cite{woodburyInvertingModifiedMatrices1950} in the other direction gives
\begin{equation}
	\label{eq:A-tilde-woodbury}
	\widetilde{A}^{-1} = A^{-1}(I - M_U(I + M_\Sigma M_V^TA^{-1}M_U)^{-1}M_\Sigma M_V^T A^{-1}).
\end{equation}
Thus, we have
\begin{align}
	\psi
	&\leq 1 + \norm{2}{(I + M_\Sigma M_V^TA^{-1}M_U)^{-1}}\norm{2}{M_\Sigma M_V^TA^{-1}} \nonumber \\
	\label{eq:psi-bound}
	&\leq 1 + \norm{2}{(I + M_\Sigma M_V^TA^{-1}M_U)^{-1}}\hat{\tau}\cond{2}{A}.
\end{align}
When \(\hat{\tau}\cond{2}{A} < 1\), we have
\begin{align*}
	\norm{2}{(I + M_\Sigma M_V^TA^{-1}M_U)^{-1}}
	&\leq \sigma_{\mathrm{min}}(I + M_\Sigma M_V^TA^{-1}M_U) ^{-1} \\
	&\leq (\sigma_{\mathrm{min}}(I) - \sigma_{\mathrm{max}}(M_\Sigma M_V^TA^{-1}M_U)) ^{-1} \\
	&\leq (1 - \hat{\tau}\cond{2}{A})^{-1},
\end{align*}
and so,
\begin{equation}
	\label{eq:psi-bound-general}
	\psi \leq 1 + \frac{\hat{\tau}\cond{2}{A}}{1 - \hat{\tau}{\cond{2}{A}}} = \frac{1}{1 - \hat{\tau}{\cond{2}{A}}}.
\end{equation}
Note that this is equivalent to the bound obtained by direct perturbation analysis of the singular values~\cite[Thm.~4.1]{lindquistUsingAdditiveModifications2023}.
Thus, if \(\hat{\tau}\cond{2}{A} \leq 0.5\), we have \(\psi \leq 2\).
Furthermore, \(M_V^T\) and \(M_U\) will select just a small portion of the inverse, which should usually reduce the \(\norm{2}{A^{-1}}\) term in the capacitance matrix; however, this is hard to quantify as a general bound. 
Similar logic implies that whenever \(\hat{\tau}\norm{2}{A}\norm{2}{(A_{1:k,1:k})^{-1}} \leq 0.5\), we have
\[
	(\widetilde{A}_{1:k,1:k})^{-1} \leq 2(A_{1:k,1:k})^{-1},
\]
which in turn limits the growth factor of \(\widetilde{A}\).
Unfortunately, it is possible for \(A\) to have both small growth factor and a leading principal submatrix with small singular value, opening the potential for the modifications to significantly increase the growth factor.

While \(\psi\) is difficult to bound in the general case, it can be easily bounded for a particular class of well behaved modifications.
Rearranging \cref{eq:A-tilde-woodbury} and taking the norm gives
\begin{align*}
	\psi
	&\leq 1 + \norm{2}{M_\Sigma^{1/2}}\norm{2}{(I + M_\Sigma^{1/2}M_V^TA^{-1}M_UM_\Sigma^{1/2})^{-1}}\norm{2}{M_\Sigma^{1/2}}\norm{2}{M_V^TA^{-1}} \\
	&\leq 1 + \norm{2}{(I + M_\Sigma^{1/2}M_V^TA^{-1}M_UM_\Sigma^{1/2})^{-1}}\hat{\tau}\cond{2}{A}.
\end{align*}
So, let \(F = M_\Sigma^{1/2}M_V^TA^{-1}M_UM_\Sigma^{1/2}\) and note that for any conformal \(x\),
\begin{align*}
	\norm{2}{(I+F)x}^2 = x^H(I+F)^H(I+F)x = x^Hx + x^H(F^HF + F + F^H)x.
\end{align*}
Thus, \(\norm{2}{(I+F)^{-1}} \leq 1\) if \(F^HF + F + F^H\) is positive semidefinite.
Therefore, when \(F^HF + F + F^H\) is positive semidefinite,
\begin{equation}
	\label{eq:psi-bound-spd}
	\psi \leq 1 + \hat{\tau}\cond{2}{A}.
\end{equation}
Additionally, this bound is smaller than \cref{eq:psi-bound-general} even when \(\hat{\tau}\cond{2}{A} < 1\).

Notably, if the Hermitian part of \(F\) is positive semidefinite, then \(F^HF + F + F^H\) is the sum of positive semidefinite matrices, and must be positive semidefinite itself.
Furthermore, because \(M_\Sigma\) is diagonal with positive values, \(F\) has positive semidefinite Hermitian part if and only if \(M_V^TA^{-1}M_U\) does.
Also, if a matrix is normal and has only eigenvalues with non-negative real part, then its Hermitian part is positive semidefinite. 
So, for these classes of \(F\) matrices, \(\psi\) is bounded, even when \(\hat{\tau}\cond{2}{A} \geq 1\).

The most important application of this result is for Hermitian positive or negative definite matrices.
Because \(M_U\) and \(M_V\) are singular vectors of diagonal blocks and block Gaussian Elimination preserves definiteness, \(M_V = M_U\) or \(-M_V = M_U\), respectively, for a positive or negative definite \(A\).
Hence, in both cases \(F\) is Hermitian positive definite, and \cref{eq:psi-bound-spd} holds.


We next look at using the properties of \ac{beam} to bound
\(\norm{2}{\widetilde{A}^{-1}}\) independently of \(\norm{2}{A^{-1}}\).
Note that the block triangular structure of the \(L\) and \(R\) factors imply that the determinant of \(\widetilde{A}\) is the product of the singular values of all the diagonal blocks.
Thus, the quotient formula for determinants~\cite[Thm.~1.4]{zhangSchurComplementIts2005} provide a finite bound for \(\psi\) and the growth factor for general matrices, independent of the condition number of \(A\).
Unfortunately, this bound is exponential, limiting its use as a practical bound.

\begin{theorem}
	\label{thm:sval-bound}
	Let \(n_k\) be the dimension of the \(k\)th diagonal block, and let \(n_{1:k}\) be the dimension of the leading principal submatrix containing the first \(k\) diagonal blocks.
	Assume that \(\sigma_{\mathrm{min}}(\widetilde{A}_{k,k}^{(k)}) \geq \tau = \hat{\tau}{\norm{2}{A}}\).
	Then,
	\begin{align*}
		\norm{2}{\widetilde{A}^{-1}}
		&\leq \left(\frac{\norm{2}{\widetilde{A}}}{\norm{2}{A}}\right)^{\!\!n-1} \!\hat{\tau}^{-n}\cond{2}{A}^{-1}\norm{2}{A^{-1}}, \\
		\cond{2}{\widetilde{A}}
		&\leq \left(\frac{\norm{2}{\widetilde{A}}}{\norm{2}{A}}\right)^{\!\!n} \!\hat{\tau}^{-n}, &\text{and} \\
		\norm{2}{A}\norm{2}{\widetilde{A}_{1:k, 1:k}^{-1}}
		&\leq \left(\frac{\norm{2}{\widetilde{A}_{1:k, 1:k}}}{\norm{2}{A}}\right)^{\!\!n_{1:k}-1} \!\hat{\tau}^{-n_{1:k}}.
	\end{align*}
	Recall that \(\norm{2}{\widetilde{A}} / \norm{2}{A} \leq (1 + \hat{\tau})\).
\end{theorem}
\begin{proof}
	Repeated application of the quotient formula for Schur complements~\cite[Thm. 1.4]{zhangSchurComplementIts2005} gives
	\begin{align*}
		\det(\widetilde{A}_{1:k, 1:k})
		&= \prod_{i=1}^k \det(\widetilde{A}_{1:i,1:i} / \widetilde{A}_{1:i-1, 1:i-1}).
	\end{align*}
	By the selection of the modifications and the fact that the absolute determinant equals the product of the singular values,
	\[
		|\det(\widetilde{A}_{1:i,1:i} / \widetilde{A}_{1:i-1, 1:i-1})| \geq \tau^{n_k}.
	\]
	Furthermore,
	\[
		|\det(\widetilde{A}_{1:k, 1:k})| = \prod_{j=1}^{n_{1:k}} \sigma_j(\widetilde{A}_{1:k, 1:k}).
	\]
	Thus, multiplying both the numerator and denominator by \(|\det(\widetilde{A}_{1:k, 1:k})|\) gives
	\begin{align}
		\nonumber
		\norm{2}{(\widetilde{A}_{1:k, 1:k})^{-1}} = \sigma_{\textrm{min}}(\widetilde{A}_{1:k, 1:k})^{-1}
		&= \frac{\prod_{j=1}^{n_{1:k}-1} \sigma_j(\widetilde{A}_{1:k, 1:k})}{\prod_{i=1}^k |\det(\widetilde{A}_{1:i,1:i} / \widetilde{A}_{1:i-1, 1:i-1})|} \\
		\label{eq:inv-norm-main-inequality}
		&\leq \frac{\norm{2}{\widetilde{A}_{1:k, 1:k}}^{n_{1:k}-1}}{\hat{\tau}^{n_{1:k}}\norm{2}{A}^{n_{1:k}}} \\
		\nonumber
		&= \left(\frac{\norm{2}{\widetilde{A}_{1:k, 1:k}}}{\norm{2}{A}}\right)^{\!\!n_{1:k}-1} \!\hat{\tau}^{-n_{1:k}} \norm{2}{A}^{-1}.
	\end{align}
	From here, the desired inequalities easily follow.
	(The first two use \(k = n_t\).)
\end{proof}

For common behavior, this theorem is very pessimistic.
Looseness in \cref{eq:inv-norm-main-inequality} is the primary source of pessimism.
First, in the numerator, the inequality can only be tight for matrices
in which all the singular values are equal except the smallest one.
However, for most matrices in practice (particularly those that are ill-conditioned), many of the singular values are much smaller than the maximum.
Second, in the denominator it is unlikely that most of the singular
values of every diagonal block will be less than or equal to the
modification threshold.
Furthermore, the tightness of these two inequalities are somewhat at odds: the former corresponds to the columns being highly independent of each other, while the latter corresponds to the columns of the leading principal submatrices being either highly dependent or of small magnitude.
However, this tension is not enough to still prevent problematically large growth, as shown below for a matrix of Zielke.

To better understand the achievable worst case growth (and by extension, the singular values of the leading submatrices),
we consider the growth incurred when factoring the nonsymmetric matrix of Zielke
with \(a=0\)~\cite{zielkeTestmatrizenMitMaximaler1974}.
Previous experiments factoring this matrix with \ac{beam} resulting in large enough growth to cause numerical overflow in double precision~\cite{lindquistUsingAdditiveModifications2023}.
Zielke's nonsymmetric matrix takes the form
\[
	\begin{bmatrix}
		0 & 1 & 1 & \dots & 1 & 1 \\
		0 & 0 & 1 & \dots & 1 & 1 \\
		0 & 0 & 0 & \dots & 1 & 1 \\
		\vdots & \vdots & \vdots & \ddots & \vdots \\
		0 & 0 & 0 & \dots & 0 & 1 \\
		-1 & 0 & 0 & \dots & 0 & 0 \\
	\end{bmatrix}.
\]
Note that for large \(n\), the smallest singular value is approximately \(1/2\), while the largest is approximately \(2n/\pi\)~\cite{elkiesAnswer2normUpper2011}.
As can be seen by the first diagonal element, pointwise non-pivoted LU breaks down (i.e., it effectively has an infinite growth factor).
So, for this case \ac{beam} is technically an improvement but, as seen in the following result, insufficient for providing a meaningful solution.
Note that this theorem is based on \(\tau\), not \(\hat{\tau}\); when \(\hat{\tau} \ll \tfrac{\pi}{4n}\), the assumption is satisfied.

\begin{theorem}
	\label{thm:zielke-growth}
	Let \(0 < \tau \leq \tfrac{1}{2}\) and assume all of the blocks are at least of size 2.
	Then, applying the \ac{beam} algorithm to the Zielke matrix results in a growth of \(\blockgrowth{\maxrm} = \tau^{1-n_t}\)
\end{theorem}
\begin{proof}
	Consider the first diagonal block.
	The last row and first column are zero, so one singular value is zero with left and right singular vectors of \(e_n\) and \(e_1\), respectively.
	Removing that row and column gives an upper-triangular matrix of all 1's, the inverse of which is an upper-bidiagonal matrix with 1's on the diagonal and -1's on the off-diagonal~\cite{elkiesAnswer2normUpper2011}.
	Thus, the remaining singular values are at least \(1/2\).
	So, only one singular value will be modified.

	By treating the modified singular value separately, it can easily be seen that the inverse of the modified diagonal block is
	\[
		\widetilde{A}_{kk}^{-1} = \begin{bmatrix}
			0 & \tau^{-1} \\
			\mathrm{tridiag}(0, 1, -1) & 0
		\end{bmatrix}.
	\]
	Because of the lower-triangular zeros, all by the last row of the trailing matrix will remain unchanged by the Schur-complement; on the other hand, all of the elements of the last row become \(\tau^{-1}\).
	So, the next diagonal block will require the same additive modification and the Schur-complement will update each entry of the last row to
	\[
		\tau^{-1} - [\tau^{-1}, \dots, \tau^{-1}]\widetilde{A}_{kk}^{-1}[1,\dots,1]^T = -\tau^{-2}.
	\]
	By repeating this process, the entries of the last row of the trailing matrix will be \(-(-\tau^{-1})^k\) after \(k\) diagonal blocks.
	Thus, the max-norm growth is \(\tau^{-(n_t-1)}\).
\end{proof}

\subsection{Experimental results}
\label{sec:svals:experiments}

Due to the pessimistic theoretical results, we experimentally tested the effect of additive modifications on the smallest singular value (and thus the condition number) as well as the growth factor.
The experimental results suggest that, in at least in the tested cases, both quantities are improved by the modifications.
Similar to the experiments in \cref{sec:block-lu:experiments}, our code and data can be found in the supplemental materials.
Furthermore, all of the tests used double precision and an algorithmic block size of 4.

We first tested the general affect of various tolerance on the smallest singular value and condition number, particularly for when \(\hat{\tau}\cond{2}{A} \geq 1\).
The matrix was constructed with singular values geometrically distributed from 1 to \(10^{-1}\) and Haar-distributed singular vectors, then QR and row-pivoting were applied to the first 8 rows to zero the leading block without affecting the singular values.
This leading zero block helps ensure modifications were applied, regardless of the tolerance.
(Note that because \(\norm{2}{A}=1\), we have \(\tau = \hat{\tau}\).)
\Cref{fig:beam-ssv} shows the effect of \(\tau\) on the smallest singular value of the modified matrix, \(\widetilde{A}\).
\begin{figure}
	\centering
	\begin{tikzpicture}
\begin{loglogaxis}[
        width=5in, height=3.3in,
        xlabel={\(\tau\)}, xmin=1e-11, xmax=1e1,
        xtick={1e-10, 1e-8, 1e-6, 1e-4, 1e-2, 1e-0},
        ylabel=Smallest singular value of \(\widetilde{A}\), ymin=2e-9, ymax=5,
        ytick={1e0, 1e-2, 1e-4, 1e-6, 1e-8, 1e-10},
    ]
    \addplot[dotted, black] coordinates {(1e-12, 1e-12) (1e2, 1e2)};
    \addplot[dotted, black] coordinates {(1e-12, 1e-8) (1e2, 1e-8)};
    \addplot[mark=*, nodes near coords, point meta=explicit symbolic] coordinates {(1.0, 0.6666887875437577) [1000] (0.1, 0.014088499520961886) [1000] (0.010000000000000002, 0.00029782743326113185) [883] (0.001, 1.5553069718946665e-5) [704] (0.0001, 1.0118297903618087e-6) [537] (1.0e-5, 7.052680096293459e-8) [380] (1.0e-6, 4.50694403483604e-9) [193] (1.0000000000000001e-7, 8.274135306326666e-9) [38] (1.0e-8, 1.0010549360009612e-8) [7] (1.0e-9, 1.0000065621748586e-8) [4] (1.0e-10, 1.0000008763872324e-8) [4] };
\end{loglogaxis}
\end{tikzpicture}
	\vspace{-3em}
	\caption{The effect of varying the tolerance on the smallest singular value of the modified matrices.  Labels indicate the number of modifications applied by \ac{beam}.}
	\label{fig:beam-ssv}
\end{figure}
The dotted lines indicate the original singular value of \(10^{-8}\) and where the modified singular value would equal \(\tau\).

Note that the smallest singular value is increased when \(\tau > \cond{2}{A}^{-1}\), and even when reduced, it is within a factor of 2 of the original value.
In other words, the matrix's conditioning is either improved or not meaningfully affected, which is significantly better than the worse case bound provided by \cref{thm:sval-bound}.
This suggests that, in practice, the restriction that  \(\hat{\tau}\cond{2}{A} < 1 \) may be less critical than the theoretical analysis suggests~\cite{lindquistUsingAdditiveModifications2023}.

Unfortunately, a large number of modifications were needed to provide much improvement in the conditioning, which would result in a large overhead for applying the Woodbury formula.
(Furthermore when \(\hat{\tau}\cond{2}{A} \geq 1\), iterative refinement without the Woodbury formula is not expected to converge, both theoretically~\cite{carsonAcceleratingSolutionLinear2018} and experimentally~\cite{lindquistUsingAdditiveModifications2023}.)
This suggests that, even if \(\hat{\tau}\cond{2}{A} < 1 \) is not necessary for the stability of \(\widehat{A}\), it will likely be limited by practical constraints.
These large numbers of corrections mirror some of the matrices in the previous \ac{beam} experimental results~\cite{lindquistUsingAdditiveModifications2023}.

The fact that \ac{beam} can increase the smallest singular value for the whole matrix suggests it can also improve the singular value of the leading principal submatrices.
Improving those singular values, combined with \cref{thm:schur-growth-bound}, suggests a more robust ability to dampen growth than the algorithm's explicit greedy minimization of growth might be expected to achieve.
Thus, \cref{fig:beam-growth} plots the growth factors for the factorizations from \cref{fig:beam-ssv}.
\begin{figure}
	\begin{tikzpicture}
\begin{loglogaxis}[
        width=5in, height=3.3in,
        xlabel={\(\tau\)}, xmin=1e-11, xmax=1e1,
        xtick={1e-10, 1e-8, 1e-6, 1e-4, 1e-2, 1e-0},
        ylabel=Growth factor, ymin=0, ymax=1e11,
        ytick={1e0, 1e2, 1e4, 1e6, 1e8, 1e10},
    ]
    \addplot[dotted, black] coordinates {(1e-12, 1e12) (1e2, 1e-2)};
    \addplot[mark=*] coordinates {(1.0, 1.0) (0.1, 1.326645952944607) (0.010000000000000002, 43.87967492787612) (0.001, 146.1016982928768) (0.0001, 1526.4356626709882) (1.0e-5, 41493.24989641434) (1.0e-6, 31034.66519292423) (1.0000000000000001e-7, 310346.2209852347) (1.0e-8, 3.103461779057607e6) (1.0e-9, 3.1034617359796256e7) (1.0e-10, 3.103461731671843e8) };
\end{loglogaxis}
\end{tikzpicture}
	\vspace{-3em}
	\caption{The effect of varying the tolerance on the growth factor for the matrices shown in \cref{fig:beam-ssv}.}
	\label{fig:beam-growth}
\end{figure}
The dotted line plots the reciprocal of \(\tau\) as a comparison point.
Note that the growth appears to be almost proportional to \(\tau^{-1}\).
Recall that, because the leading diagonal block of \(A\) is zero, the corresponding block of \(\tilde{A}\) is \(\tau I\).
So, similar to the proof of \cref{thm:schur-growth-bound}, the growth introduced by the first step is bounded by \(\tau^{-1}\).
Furthermore, when \(\tau\) is less than the original smallest singular value, increasing \(\tau\) meaningfully reduces the growth factor while retaining a small number of modifications.

\section{Modification-Free BEAM}
\label{sec:mod-free-beam}
Mirroring the question of how modifications affect singular values is the question of when \ac{beam} will not apply any modifications.
Ideally, no modifications will be applied for matrices for which regular block LU is stable to avoid the unnecessary cost of correcting for the modifications.

Note that block inversion implies
\[
	\norm{2}{(A_{1:k,1:k})^{-1}} \geq \norm{2}{(A_{k,k}^{(k)})^{-1}},
\]
and thus,
\begin{equation}
	\label{eq:lps-inv-vs-diag-block}
	\sigma_{\mathrm{min}}(A_{1:k,1:k}) \leq \sigma_{\mathrm{min}}(A_{k,k}^{(k)}).
\end{equation}
Hence, if, for all \(k\), we have \(\tau \leq \sigma_{\mathrm{min}}(A_{1:k,1:k})\), then \ac{beam} will not modify the matrix.
The condition is equivalent to \(\norm{2}{A}\norm{2}{(A_{1:k,1:k})^{-1}} \leq \hat{\tau}^{-1}\); recall that \cref{thm:schur-growth-bound} says that
\[
	\blockgrowth{2} \leq 1 + \max_{1\leq k \leq n_t-1}\norm{2}{A}\norm{2}{(A_{1:k,1:k})^{-1}}.
\]
This reinforces the notion that the behavior of \(\max_{1\leq k \leq n_t-1}\norm{2}{A}\norm{2}{(A_{1:k,1:k})^{-1}}\) is important for understanding this method.
Unfortunately, because of the directions of the implications, a matrix either having small growth or that is not modified does not provide an upper bound for \(\norm{2}{(A_{1:k,1:k})^{-1}}\) nor, by extension, anything meaningful about the other property.

Quantifying the looseness of the inequality in
\cref{eq:lps-inv-vs-diag-block} would provide significant insight into
the element growth in \ac{beam}.
Unfortunately, this is difficult in general.
Consider the triangular matrix, \(T\), with 1's on the diagonal and -1's on the off-diagonal.
Turing showed that \(\norm{\maxrm}{T^{-1}} = 2^{n-1}\)~\cite{turingRoundingoffErrorsMatrix1948}.
Thus, the smallest singular value is less than \(2^{1-n}\); however, LU with a block size of 1 (i.e., pointwise LU) results in \(A_{k,k}^{(k)} = 1\) for all \(k\).
So, \cref{eq:lps-inv-vs-diag-block} can be extremely loose for some matrices.\footnotemark
\footnotetext{This is also the cause of exponential growth for pointwise LU with partial pivoting on Wilkinson's notorious matrix~\cite[p.~212]{wilkinsonAlgebraicEigenvalueProblem1965}.}

While \((A_{1:k,1:k})^{-1}\) is difficult to work with in general, we can analyze it for specific classes of matrices.
First we consider strictly diagonally dominant matrices.
Let \(A\) be diagonally dominant by columns, and let \(\delta_c\) be it's dominating factor, that is
\[
	\delta_c = \min_{1\leq i\leq n} |A[i,i]| - \sum_{j\neq i} |A[j,i]|
\]
for dominance by columns.
So, the dominating factor of the leading principal submatrices is also at least \(\delta_c\).
Then, by a result of Varah~\cite{varahLowerBoundSmallest1975}, we have for all \(1\leq k\leq n\)
\begin{align*}
	\norm{2}{A[\range{1}{k}, \range{1}{k}]^{-1}} &\leq \sqrt{n}\norm{1}{A[\range{1}{k}, \range{1}{k}]^{-1}} \leq \delta_c^{-1}\sqrt{n}.
\end{align*}
Thus, if \(\tau \leq \delta_c^{-1}\sqrt{n}\), no modifications will be applied.
Similarly, if \(A\) is instead dominant by rows, no modification will be applied when \(\tau \leq \delta_r^{-1}\sqrt{n}\).
Finally, if \(A\) is dominant by both rows and columns, the condition is \(\tau \leq \delta_c^{-1/2}\delta_r^{-1/2}\).
This result can be generalized to all H-matrices.
If \(A\) is an H-matrix, then there exist diagonal matrices \(D_1, D_2\) with elements in the half-open interval \((0, 1]\) such that \(D_1A\) is strictly diagonally dominant by columns and \(AD_2\) is strictly diagonally dominant by rows~\cite{plemmonsMmatrixCharacterizationsInonsingularMmatrices1977}.
Let \(\delta_c^D\) be the column dominating factor of \(D_1A\) and \(\delta_r^D\) be the row dominating factor of \(AD_2\).
Then, by a generalization of Varah's result~\cite{vargaDiagonalDominanceArguments1976}, no modifications will occur when
\begin{align*}
	\tau &\leq {\delta_c^D}^{-1}\sqrt{n}, \\
	\tau &\leq {\delta_r^D}^{-1}\sqrt{n}, \qquad \text{or}\\
	\tau &\leq (\delta_c^D\delta_r^D)^{-1/2}.
\end{align*}
Note that this also improves the criteria for matrices strictly diagonally dominant by columns when
\[
	|A[i, i]| - \sum_{j\neq i}|A[j, i]|
\]
varies drastically between columns and correspondingly for dominance by rows.

Note that Varah's result also addresses block diagonally dominant matrices~\cite{varahLowerBoundSmallest1975}, allowing our analysis to also extend to those matrices.
However, note that Varah's Cor.~3 has a typo and the \(\infty\)-norms in the definition of \(\beta\) should instead be \(1\)-norms.\footnotemark
\footnotetext{The corollary as written is contradicted by
\(
	\begin{bsmallmatrix}
		1.0\, & \,0.0\, & \,0.0\, & \,0.4 \\
		0.0\, & \,1.0\, & \,0.0\, & \,0.4 \\
		0.4\, & \,0.0\, & \,1.0\, & \,0.0 \\
		0.4\, & \,0.0\, & \,0.0\, & \,1.0 \\
	\end{bsmallmatrix}
\)
with a block size of 2.}
Let \(\Delta_c\) be the block dominance factor for columns with the \(1\)-norm and let \(\Delta_r\) be for rows with the \(\infty\)-norm.
Then, no modifications are applied when
\begin{align*}
	\tau &\leq \sqrt{n}(\Delta_c)^{-1} && \text{if block dominant in the \(1\)-norm by columns,} \\
	\tau &\leq \sqrt{n}(\Delta_r)^{-1} && \text{if block dominant in the \(\infty\)-norm by rows, or} \\
	\tau &\leq (\Delta_c\Delta_r)^{-1/2} && \text{if block dominant by both.}
\end{align*}
Furthermore, we expect the analysis of H-matrices to similarly generalize to block H-matrices.

Finally, for a symmetric positive definite matrix \(A\), no modification will occur when \(\hat{\tau}\cond{2}{A} \leq 1\).
Combining Cauchy's interlace theorem with a similar theorem for Schur complements~\cite[Cor. 2.4]{zhangSchurComplementIts2005}, we have
\(
	\sigma_i(A_{kk}^{(k)}) \geq \sigma_{k}(A_{1:k, 1:k}) \geq \sigma_n(A)
\).
Hence, \(\hat{\tau}\cond{2}{A} \leq 1\) implies
\[
	\tau \leq \cond{2}{A}^{-1}\norm{2}{A} \leq \sigma_n(A) \leq \sigma_i(A_{kk}^{(k)}).
\]
Therefore, no modification will be applied.

This analysis shows that \ac{beam} avoids modifying at least some types of well-behaved matrices that can be safely factored by regular block LU, meaning the overhead of using \ac{beam} instead of regular block LU will be negligible on those classes of matrices.
Note that all of the analysis in the section is built around \cref{eq:lps-inv-vs-diag-block} bounding the singular values of individual blocks by the smallest singular value of the entire leading principal submatrix.
However, that inequality will often be very loose, so the class of modification-free matrices will be larger than is proven here.


\section{Conclusions}

To better understand the \ac{beam} algorithm, we have analyzed both the numerical stability of the underlying block LU factorization and the effects of the modifications.
For block LU, we developed a generalized definition of growth factor that better adapts to the block structure of the factorization.
Furthermore, we expanded the applicability and improved the error bounds for block LU from Demmel et al.~\cite{demmelStabilityBlockLU1995}, in particular the pointwise growth was replaced by the new blockwise growth with the exponent reducing from cubic to quadratic.
Next, we considered the effect of the modifications on the factorization.
Notably, we bound the growth factor and the condition number of the modified matrix, albeit exponentially in \(n\).
This is complemented with an example matrix with growth exponential in \(n_t\).
Finally, we demonstrated that \ac{beam} will not modify some types of matrices that block LU is numerical stable on, including block and pointwise diagonally dominant matrices and moderately conditioned symmetric positive definite matrices.


\bibliographystyle{siamplain}
\bibliography{dense-factorization}

\end{document}